\documentclass[10pt]{article}

\usepackage[numbered]{bookmark}
\usepackage{graphicx}
\usepackage{amsmath}
\usepackage{amsfonts}
\usepackage{amsthm} 
\usepackage{amssymb}
\usepackage{cite}
\usepackage{mathrsfs}
\usepackage{changepage} 
\usepackage{mathtools}
\usepackage{mathabx}
\usepackage{verbatim} 
\usepackage{enumerate}

\newtheorem{thm}{Theorem}[section]

\newtheorem{lem}[thm]{Lemma}
\newtheorem{prop}[thm]{Proposition}
\newtheorem{cor}[thm]{Corollary}

\newcommand{\cartesian}[1]{(G_{\omega, #1} \times G_{\omega, #1}, \mu_{\omega, #1} \times \mu_{\omega, #1})}
\newcommand{\transf}[1]{(G_{\omega, #1}, \mu_{\omega, #1})}
\newcommand{\tmap}[1]{T_{\omega, #1}}
\newcommand{\map}[1]{G_{\omega, #1}}
\newcommand{\meas}[1]{\mu_{\omega, #1}}

\newcommand{\R}{\mathbb{R}}
\newcommand{\Z}{\mathbb{Z}} 
 
\newcommand{\N}{\mathbb{N}}
 
\newcommand{\T}{\mathbb{T}}
\newcommand{\Pa}{\mathcal{P}}
\newcommand{\Qa}{\mathcal{Q}}
\newcommand{\Pt}{{\mathcal{P}^\twr}}
\newcommand{\twr}{\mathcal{T}}
\newcommand{\twrb}{\twr(F, h)}
\newcommand{\MixColl}{\mathcal{C}_{\text{Mix}}}
\newcommand{\hone}{\hspace{1cm}}
\newcommand{\FP}{\mathcal{P}_f(X)}
\newcommand{\px}{x^\Pa}

\newcommand{\pnx}{x^{\Pa, n}}
\newcommand{\pny}{y^{\Pa, n}}

\newcommand{\TSet}{\mathcal{T}(X)}
\newcommand{\DS}{(X, \mu, \mathcal{E}, T)}
\newcommand{\DSS}{(X, T)}
\usepackage{color} 
\begin{document}

\author{Frank Trujillo}
\date{}

\title{Smooth mixing  transformations with loosely Bernoulli cartesian product}
\maketitle

\begin{abstract}

A zero-entropy system is said to be loosely Bernoulli if it can be induced from an irrational rotation of the circle. We provide a criterion for zero-entropy systems to be loosely Bernoulli that is compatible with mixing. Using these criteria, we show the existence of smooth mixing  zero-entropy loosely Bernoulli transformations whose cartesian product with themselves is loosely Bernoulli.

\end{abstract}

\section{Introduction}

One of the main goals of ergodic theory is to describe invariants and models which classify broad classes of systems with respect to certain equivalence relations. Some of these relations, such as metric isomorphism, may prove too restrictive, whereas some others, such as orbit equivalence, may prove too flexible (recall that by Dye's Theorem \cite{dye_groups_1959} any two ergodic measure preserving transformations on non-atomic spaces are orbit-equivalent). Kakutani equivalence, initially introduced by Kakutani in \cite{kakutani_induced_1943}, can be thought of as an intermediate notion of classification. Two ergodic measure preserving transformations $(X, T)$ and $(Y, S)$ are said to be \textit{Kakutani equivalent} if there exist positive measure sets $A \subset X,$ $B \subset Y$ such that the induced maps $T_A$, $S_B$ are metrically isomorphic. \\ 

Particularly important in the theory of Kakutani is the equivalence class of irrational circle rotations. This class was initially considered by Katok \cite{katok_time_1975} under the name of \textit{standard} or  \textit{loosely Kronecker} transformations, and  independently introduced by Feldman \cite{feldman_newk-automorphisms_1976}, who coined the term \textit{loosely Bernoulli (LB)}  to denote the ergodic automorphisms that are:  Kakutani-equivalent to an irrational circle rotation in the zero entropy case, or to a Bernoulli shift in the positive entropy case. \\

Most of the existing criteria for loose-Bernoullicity rely on the presence of particular periodic approximations for the underlying transformation (a remarkable exception to this is the horocycle flow \cite{ratner_horocycle_1978}). Unfortunately, these approximations do not easily coexist with mixing and some of these criteria directly forbid it (see for example \cite[Proposition 3.5]{katok_combinatorial_2003}). Furthermore, the coexistence of mixing with certain periodic approximations, such as rank one transformations, becomes even more difficult in the smooth setting (see \cite{fayad_rank_2005}). \\
 
In this work we aim to provide criteria, compatible with mixing, to guarantee the loose-Bernoullicity of a zero-entropy system and of its cartesian product with itself. Our criteria will be based on localized periodic approximations inspired by \cite{fayad_smooth_2006}. We then show that these criteria are verified by a class of infinitely differentiable zero-entropy mixing transformations in $\T^3$. These systems are obtained by modifying the class of mixing transformations introduced in \cite{fayad_analytic_2002}.

\subsection{Periodic approximations}

Approximation of measure preserving systems by periodic transformations was first proposed in the forties by Halmos \cite{halmos_approximation_1944} who provided a general framework for this theory and used it to prove the genericity of several ergodic properties. Halmos \cite{halmos_general_1944} successfully applied his method to prove that the set of weak mixing transformations is a $G_\delta$ dense set in the space of measure preserving automorphisms endowed with the weak topology. Shortly after, and by similar methods, Rokhlin \cite{rohlin_general_1948} proved that the set of mixing transformations is of the first category under the same weak topology. \\

Two decades later, the theory of periodic approximations regained strength after the works of Katok and Stepin, \cite{katok_entropy_1967}, \cite{katok_approximation_1966}, \cite{katok_approximations_1967}, \cite{stepin_spectrum_1967}. These works laid out the bases of the so called \textit{method of approximations} and introduced a different point of view on periodic approximations which, in contrast to Halmos' approach, allowed the study of individual automorphisms by approximation. Katok \cite{katok_time_1975} provided periodic approximation criteria for standardicity. He proved that a system with \textit{good cyclic} approximations must be standard and that the existence of \textit{excellent linked approximations of type (h, h + 1)} for a given transformation guarantee the standardicity of its cartesian product with itself. For definitions and proofs of these statements we refer the interested reader to \cite{katok_time_1975}, \cite{katok_combinatorial_2003}.

\subsection{Cartesian products}

Several notions related to periodic approximations, such as rank one, finite rank, local rank one and very weakly Bernoulli, are known to imply loose-Bernoullicity (see \cite{ferenczi_systems_1997}, \cite{ornstein_equivalence_1982}). All of these properties, which attempt to capture precise aspects of the long time behaviour of measure preserving transformations, provide important tools for classification and illustrate the richness of existing phenomena in abstract ergodic theory. \\

Given a measure preserving automorphism satisfying any of the ergodic properties mentioned above, it is not necessarily true that its cartesian product with itself verifies the same property. In fact, Ornstein, Rudolph and Weiss \cite{ornstein_equivalence_1982} constructed a rank one transformation whose product with itself is not LB. Ratner showed that any non trivial transformation $T$ in the horocycle flow is LB \cite{ratner_horocycle_1978} but its cartesian product $T \times T$ is not LB \cite{ratner_cartesian_1979}.\\

Examples of weak mixing, LB, zero-entropy transformations with LB cartesian product with themselves were first proposed by Katok (a detailed account of his construction can be found in \cite{gerber_zero-entropy_1981}). Katok's example was later modified by Gerber \cite{gerber_zero-entropy_1981} who showed the existence of a zero-entropy LB transformation which is mixing of all orders and whose cartesian product with itself is LB. More recently, Gerber and Kunde \cite{gerber_smooth_2018} provided examples of smooth LB transformations whose cartesian product with themselves is LB. The loose-Bernoullicity of their examples relies on a slight modification of the aforementioned loose-Bernoullicity criterion proved by Katok \cite{katok_time_1975}, namely, the existence of excellent linked approximations of type $(h, h + 1)$. As in Katok's case, their criterion also guarantees that the transformation is weak mixing but not mixing (see \cite[Proposition 4.1]{gerber_smooth_2018} for the exact statement of the criterion).
 
\subsection{Plan of the work}

This work is organized as follows. In Section \ref{sc: LB_preliminaries}, we introduce the notations and basic objects we use along the paper. Proposition \ref{thm: LB_process_criterion} is an adaption of the loose-Bernoullicity criterion (Proposition \ref{thm: LB_measure_criterion}) proven by Janvresse and de la Rue in \cite{janvresse_pascal_2004}. \\

In Section \ref{sc: LB_criterion}, we combine Proposition \ref{thm: LB_process_criterion} and the existence of an appropriate sequence of Rokhlin towers to prove a loose-Bernoullicity criterion for ergodic zero-entropy processes (Item 3 in Proposition \ref{thm: LB_criterion_towers}). We stress the fact that although the measures of the towers in the criterion verify certain non-summability condition, their measures are not necessarily bounded from below. As a direct consequence of Proposition \ref{thm: LB_criterion_towers} we deduce, in Theorem \ref{thm: LB_criterion_transformation}, a similar loose-Bernoullicity criterion for ergodic zero-entropy transformations. We then adapt these criteria to obtain analogous results for the cartesian product of weak mixing zero entropy transformations with themselves (Proposition \ref{thm: LB_criterion_towers_product} and Theorem \ref{thm: LB_criterion_transformation_product}). \\

Finally, in Section \ref{sc: LB_construction}, we prove the existence of a family of smooth, mixing, zero-entropy, loosely Bernoulli diffeomorphisms of $\T^3$ whose cartesian product with themselves is loosely Bernoulli (Theorem \ref{thm: mixing_LB_existence}). The Appendix contains technical tools used in the construction of these transformations. 

\section{Preliminaries}
\label{sc: LB_preliminaries}

Throughout this work $(X, \mu, \mathcal{E})$ will denote a Lebesgue measure space and $T$ will denote a zero-entropy measure preserving automorphism on $X$. If there is no risk of confusion we will denote $(X, \mu, \mathcal{E})$ simply by $X$. 

\subsection{Partitions and $\Pa$-names.}

We denote by $\FP$ the set of finite measurable partitions of $X$. We consider a \textit{partition} $\Pa \in \FP$ to be a collection of disjoint sets $\Pa = \{ P_1, \dots, P_m\}$ with a fixed ordering of its elements and whose union is $X$. Given $A \in \mathcal{E}$ the induced partition on A is defined as
\[\Pa|_A = \{ P_i \cap A \mid 1 \leq i \leq m\}.\]
Given two partitions $\mathcal{P} = \{ P_1, \dots, P_m\}, \mathcal{Q} = \{ Q_1, \dots, Q_l\}$ of $X$ we define its cartesian product $\Pa \times \Qa$ as the finite partition of $X \times X$ given by
 \[ \Pa \times \Qa = \{ P_i \times Q_j \mid 1 \leq i \leq m, 1 \leq j \leq l \}\]
and ordered lexicographically. We say that $\Pa$ \textit{refines} $\Qa$ if for every $P \in \Pa$ there exists $Q \in \Qa$ such that $P \subset Q$. We denote this relation by $\Pa \preccurlyeq \Qa$. Whenever $l = m$ we denote
\[ D(\Pa, \Qa) = \sum_{k = 1}^m \mu(P_k \Delta Q_k).\] 
 Given $\Pa \in \FP$ the pair $(T, \Pa)$ is called a \textit{process}. For $x \in X$ we define its $\mathcal{P}$-\textit{name} as the unique vector $\px \in \Z^m$ satisfying for all $i \in \Z$
\[ \px_i = j \hspace{0.5cm}\text{ if and only if }\hspace{0.5cm} T^i(x) \in P_{j},\] 
where $\px_i$ denotes the $i$-th coordinate of $\px$. Recall that whenever $\Pa$ is a \textit{generating partition} of $\mathcal{E}$, i.e. $\mathcal{E}$ is the smallest $T$ invariant $\sigma$-algebra containing $\Pa$, the system $\DSS$ is equivalent (in a measure theoretic sense) to the shift over the space of $\Pa$-names of $X$. Given $n \in \N$ we define the $\mathcal{P}$-$n$-\textit{name} associated to $x$ as the $n$-tuple
\[ \pnx = (\px_i)_{0 \leq i < n}.\]
For a partition with exactly $m$ elements, the $\mathcal{P}$-$n$-names are naturally a subset of the space of \textit{words of length $n$} over the finite alphabet $\{1, 2, \dots, m\}$ which we denote by $W^n_m$. The \textit{length} of a word $w$ is denoted by $|w|$. The \textit{concatenation} of two words $v, w$ is denoted by $v \cdot w$. To simplify the notation we denote the concatenation of $k$ words $v_1, v_2, \dots, v_k$ by $\prod_{i = 1}^k v_i.$ We denote the \textit{Hamming distance} in the space of $n$-words by
\[ d_n(v, w) = \dfrac{1}{n}\#\{ 0 \leq i < n \,\mid\, v_i \neq w_i\}\] 
and the \textit{modified Hamming distance} or \textit{$\overline{f}$ distance} by
\[ \overline{f}_n(v, w) = \dfrac{n - l}{n},\]
where $l$ is the maximum natural number for which there exist sequences 
\[0 \leq i_1 < i_2 < \dots < i_l < n, \hone 0 \leq j_1 < j_2 < \dots < j_l < n,\]
obeying $v_{i_s} = w_{j_s}$ for $s = 1, \dots, l.$ Given $x, y \in X$ we denote
\[ d_n^\Pa(x, y) = d_n(\pnx, \pny), \hone \overline{f}_n^\Pa(x, y) = \overline{f}_n(\pnx, \pny).\] 
For any $W \in \mathcal{E}$ denote
\[ \overline{f}_n^\Pa (W) = \sup_{x, y \in W} \overline{f}^\Pa_n(x, y).\] 

\subsection{Loose-Bernoullicity}
\label{sc: LB}

Since we will only consider zero-entropy transformations we provide a simplified definition of loose-Bernoullicity for this particular case. For a more general definition we refer the reader to \cite{feldman_newk-automorphisms_1976} or \cite{ornstein_equivalence_1982}. In the following, we will refer to a zero-entropy measure preserving transformation $\DS$ equipped with a finite measurable partition $\Pa \in \FP$ as a \textit{zero-entropy process} which, as an abuse of notation, we denote simply by $(T, \Pa)$.\\ 

A zero-entropy process $(T, \Pa)$ is called \textit{loosely Bernoulli} (LB) if for any $\epsilon > 0$ there exists $n_0 \in \N$ such that for all $n > n_0$ there exists $x \in X$ obeying 
\[ \mu\left(\left\{ y \in X \, \Big| \, \overline{f}^\Pa_n(x, y) < \epsilon \right\}\right) > 1 - \epsilon. \]
A zero-entropy automorphism $T$ is said to be \textit{loosely Bernoulli} (LB) if the associated process $(T, \Pa)$ is LB for some generating partition $\Pa$ or, equivalently, if there exists a refining sequence of finite measurable partitions $(\Pa_n )_{n \in \N} \subset \FP$ obeying $\Pa_n \rightarrow \mathcal{E}$ such that $(T, \Pa_n)$ is LB for all $n \in \N$. Recall that $\Pa_n \rightarrow \mathcal{E}$ means that for any $A \in \mathcal{E}$ there exists a sequence $(A_n)_{n \in \N} \subset \mathcal{E}$ such that $A_n$ is $\Pa_n$-measurable and $\mu(A_n \Delta A) \rightarrow 0$. Let us point out that a transformation is LB if and only if for any finite measurable partition $\Pa$ the process $(T, \Pa)$ is LB. For proofs of this facts and a detailed account of LB transformations see \cite{ornstein_equivalence_1982}.\\

The following criterion for loose-Bernoullicity was proved by Janvresse and de la Rue in \cite{janvresse_pascal_2004}.

\begin{prop}
\label{thm: LB_measure_criterion}
Let $(T, \Pa)$ be a zero-entropy process. Suppose that for all $\epsilon > 0$ and for $\mu \times \mu$-a.e. $(x, y) \in X \times X$ there exists $l(x, y) \in \N$ such that
\[ \overline{f}_{l(x, y)}^\Pa(x, y) < \epsilon,\]
then the process $(T, \Pa)$ is loosely Bernoulli.
\end{prop}

\subsubsection{The $P(\Pa, \alpha, \delta, n)$ property}

Given $\alpha, \delta \in [0, 1)$ and $n \in \N$ we say that the transformation $T$ \text{verifies property} $P(\Pa, \alpha, \delta, n)$ or, equivalently, that the process $(T, \Pa)$ \text{verifies property $P(\alpha, \delta, n)$} if there exists $W \in \mathcal{E}$ obeying
\[\mu(W) > 1 - \delta, \hspace{1cm} \overline{f}_n^\Pa(W) < 1 - \alpha.\]
With these notations, a process $(T, \Pa)$ is loosely Bernoulli if and only if for every $\epsilon > 0$ there exists $n_0 \in \N$ such that for all $n > n_0$ the process verifies property $P(\epsilon, \epsilon, n)$. Proposition \ref{thm: LB_measure_criterion} implies the following criterion for loose-Bernoullicity. 

\begin{prop} 
\label{thm: LB_process_criterion}
Let $(T, \Pa)$ be a zero-entropy process. Suppose there exist sequences $\delta_n \searrow 0$, $\alpha_n \nearrow 1$, $(m_n)_{n \in \N} \subset \N$ such that $(T, \Pa)$ verifies $P(\alpha_n, \delta_n, m_n)$ for all $n \in \N$. Then the process $(T, \Pa)$ is loosely Bernoulli.
\end{prop}

\begin{proof}
By hypothesis, for all $n \in \N$ there exists $W_n \in \mathcal{E}$ such that
\[\mu(W_n) > 1 - \delta_n, \hspace{1cm} \overline{f}_{m_n}^\Pa(W_n) < 1 - \alpha_n. \]
Define
\[ E := \limsup_{n \in \N} (W_n \times W_n) = \bigcap_{i = 1}^{\infty} \bigcup_{n = i}^{\infty} W_n \times W_n.\]
Clearly $(\mu \times \mu)(E) = 1$. Suppose $\epsilon > 0$ and let $n_{0} \in \N$ satisfy $1 - \alpha_{n_0} < \epsilon$. Given $(x, y) \in E$ there exists $n > n_0$ such that $(x, y) \in W_n \times W_n$. Therefore
\[ \overline{f}_{m_n}^\Pa(x, y) < 1 - \alpha_n < \epsilon.\]
By Proposition \ref{thm: LB_measure_criterion} the process $(T, \Pa)$ is loosely Bernoulli.
\end{proof}

\subsection{Rokhlin towers}

 Given $F \in \mathcal{E}$ and $h \in \N$ we say that the collection 
 \[\twr = \twr(F, h) = \{F, TF,\cdots T^{h - 1}\}\]
 is a \textit{Rokhlin tower} of \textit{height} $h$ if the sets $F, TF,\dots, T^{h - 1}F$ are pairwise disjoint. The sets $T^i F$ are called \textit{levels} of the tower. The first and last elements of the collection are known as \textit{base} and \textit{roof} of the tower. The union of all levels is called the \textit{support} of the tower. By an abuse of notation $\twr$ will denote the support of the tower and also the tower itself. The measure of the support is called the \textit{size} of the tower. We denote by $\TSet$ the set of Rokhlin towers of $X$. Given $\twr \in \TSet$ we denote its height by $h_\twr$ and its size by $\mu_\twr$. \\
 
As we shall see, Rokhlin towers can be used to easily \textit{match} coordinates between two given $\Pa$-names. Given two words $v, w$, not necessarily having the same length, and a natural number $l$, we say we can match at least $l$ coordinates of $v$ and $w$ if there exist sequences
\[ 0 \leq i_1 < i_2 < \dots < i_l \leq |v|, \hspace{1cm} 0 \leq j_1 < j_2 < \dots < j_l \leq |w|,\]
obeying 
\[ v_{i_s} = w_{j_s} \hspace{1cm} \text{for } s = 1, \dots, l. \]
Whenever $|v| = n = |w|$ this is equivalent to say that $\overline{f}_n(v, w) \leq \tfrac{n - l}{n}$. To illustrate how to use Rokhlin towers to match two $\Pa$-$n$-names associated to different points, let $(T, \Pa)$ be a process, $\twrb$ be a Rokhlin tower and consider the finite measurable partition $\mathcal{Q} \in \FP$ given by
 \[ \mathcal{Q} = \twr \cup \Pa|_{\twr^c}.\]
Notice that
 \[x, y \in F \implies x^{\mathcal{Q}, h} = y^{\mathcal{Q}, h}.\] 
 Thus, for any $x, y \in X$, finding iterates 
 \[ 0 \leq i_1 < i_2 < \dots < i_L < n, \hone 0 \leq j_1 < j_2 < \dots < j_L < n\]
for which $T^{i_l}x, T^{j_l} y \in F$, means we can match at least $hL$ coordinates of $x^{\Qa, n}$ and $y^{\Qa, n}.$ Such iterates can be easily found whenever the system possesses some recurrence property. For ergodic systems the relation between $\Pa$-$n$-names and $\Qa$-$n$-names, for sufficiently large values of $n$ (see Lemma \ref{thm: names_and_towers}), is determined by how well the tower $\twr$ approximates the restricted partition $\Pa|_\twr$. Given a partition $\Pa = \{P_1, \dots, P_m\} \in \FP$ and a Rokhlin tower $\twr \in \TSet$ we say that $\twr$ is \textit{$\delta$-monochromatic with respect to $\Pa$} if 
\[ \Delta(\twr, \Pa) = \inf_{\substack {\twr \preccurlyeq \xi, \\ |\xi| = m}} \sum_{k = 1}^{m} \mu(\xi_k \,\setminus\, P_k) < \delta,\] 
where the infimum is taken over all the partitions of the support of the tower with exactly $m$ sets and whose elements are unions of levels of the tower. If 
\[\Delta(\twr, \Pa) = 0,\]
we say that the tower is \textit{monochromatic with respect to $\Pa$}. Let us point out that this is equivalent to \[\twr \preccurlyeq \Pa|_\twr.\] Notice that the infimum in the definition of $\Delta$ is always attained but not necessarily unique. We can provide an explicit expression of a minimizing partition, and thus of $\Delta(\twr, \Pa)$, as follows. Given a tower $\twr = \twrb$ and a partition $\Pa = \{ P_1, \dots, P_m\}$ define for all $0 < k < h$
\[ I^\twr_\Pa(k) = \min \left\{ 0 < j \leq m \,\bigg|\, \mu(T^k(F) \,\setminus\, P_j) = \min_{0 < i \leq m}\mu(T^k(F) \,\setminus\, P_i) \right\}.\]
If there is no risk of confusion we denote $ I^\twr_\Pa(k)$ simply by $I(k)$. Let
\[ P^\twr_j = \bigcup_{I(k) = j} T^k F \]
for all $0 < j \leq m$ and define $\Pt = \{ P^\twr_1, \dots, P^\twr_m\}$. Denote
\begin{equation}
\label{eq: difference_set}
D^\twr_\Pa = \bigcup_{k = 1}^h T^k(F) \,\setminus\, P_{I(k)}.
\end{equation}
The set $D^\twr_\Pa$ is called the \textit{difference set} of $\twr$ and $\Pa$. By construction
\[\Delta(\twr, \Pa) = \sum_{k = 1}^{m} \mu(P^\twr_k \,\setminus\, P_k) = \mu(D^\twr_\Pa).\] Notice that for all $x, y \in X \,\setminus\, D^\twr_\Pa$ 
\[ x^\Qa_0 = y^\Qa_0 \implies x^\Pa_0 = y^\Pa_0. \]
By applying the pointwise ergodic Theorem to $D^\twr_\Pa$ we obtain the following. 

\begin{lem}
\label{thm: names_and_towers}
Let $(T, \Pa)$ be an ergodic process. Suppose $\twr \in \TSet$ and denote 
\[ \Qa = \twr \cup \Pa|_{\twr^c}.\]
Given $\delta > 0$ there exists $n_0 \in \N$ and $W \in \mathcal{E}$ with $\mu(W) > 1 - \delta$ such that
\[d_n^\Pa(x, y) \leq d_n^\Qa(x, y) + 3\Delta(\twr, \Pa)\]
for all $n > n_0$ and for all $x, y \in W.$
\end{lem}

\begin{proof}
If $\Delta(\twr, \Pa) = 0$ then $ \twr \preccurlyeq \Pa|_\twr $ which yields to $d_n^\Pa(x, y) = d_n^\Qa(x, y)$ for all $n \in \N$ and all $x, y \in X$. Otherwise, assume $\Delta(\twr, \Pa) > 0$. Denote by $D_\Pa^\twr$ the difference set defined of $\twr$ and $\Pa$ as defined in (\ref{eq: difference_set}). By ergodicity, there exists $n_0 \in \N$ so that the measure of the set $W$ of points in $X$ satisfying 
\[ \#\left\{ 0 \leq i < n \,\mid\, T^i x \in D_\Pa^\twr \right\} \leq \frac{3}{2}\mu(D^\twr_\Pa)n = \frac{3}{2}\Delta(\twr, \Pa)n,\]
 for all $n > n_0$, has measure at least $1 - \delta$. Hence, for all $n > n_0$ and all $x, y \in W$
\begin{align*}
n (d_n^\Pa(x, y) - d_n^\Qa(x, y)) \leq & \#\left\{ 0 \leq i < n \,\mid\, x^\Qa_i = y^\Qa_i \text{ and } x^\Pa_i \neq y^\Pa_i \right\} \\
\leq & \#\left\{ 0 \leq i < n \,\mid\, T^i x \in D_\Pa^\twr \text{ or } T^i y \in D_\Pa^\twr \right\} \\
\leq & 3\Delta(\twr, \Pa)n
\end{align*}
\end{proof}

\section{Loose-Bernoullicity criteria}
\label{sc: LB_criterion}

\subsection{Main Proposition}

The following result is at the core of this work. It embodies the idea of using Rokhlin towers to improve existent bounds on the $\overline{f}$ distance associated to a given finite measurable partition. More precisely, given an ergodic process $(T, \Pa)$, a natural number $n$ and a positive measure set $W \in \mathcal{E}$, the presence of a sufficiently high, large and monochromatic Rokhlin tower implies the existence of a set $W_+$, whose measure can be taken as close to one as desired, such that
\[ \overline{f}^\Pa_m(W_+) < \overline{f}^\Pa_n(W) \]
for sufficiently large $m$. Proposition \ref{thm: LB_prop} provides precise estimates on $\overline{f}_m(W_+)$. Recall that for any $W \in \mathcal{E}$ we denote
\[ \overline{f}_n^\Pa (W) = \sup_{x, y \in W} \overline{f}^\Pa_n(x, y)\] 
and that a process $(T, \Pa)$ is said to verify property $P(\alpha, \delta, n)$ if there exists $W \in \mathcal{E}$ such that
\[ \mu(W) > 1 - \delta, \hone \overline{f}_n(W) < 1 - \alpha. \]
The following proof is inspired on Ferenczi's argument in \cite[Proposition 1]{ferenczi_systemes_1984}.
\begin{prop}
\label{thm: LB_prop}
Let $(T, \Pa)$ be an ergodic process verifying $P(\alpha, \delta, N)$ for some $\alpha \geq 0,$ $\delta > 0$ and $N \in \N$. Suppose there exists a Rokhlin tower $\twr \in \TSet$ obeying
\begin{equation*}
\label{eq: height_condition}
 h_\twr > \dfrac{2N}{\delta}.
 \end{equation*}
Then for any $\delta_+ > 0$ there exists $N_+ \in \N$ such that the process $(T, \Pa)$ verifies $P(\alpha_+, \delta_+, N_+)$ where
\begin{equation} 
\label{eq: new_alpha}
\alpha_+ = \alpha + \frac{\mu_\twr^2}{10}(1 - \alpha)^2 - 6\delta - 3\Delta(\twr, \Pa).
\end{equation}
\end{prop}

\begin{proof} Fix $0 < c < 1$. Let
 \[\twr_{c} = \twr(F, \lceil ch \rceil - 1)\]
 and let us denote its size by $\mu_c$. Let $N_0$ sufficiently large so that the set $W_0$ of points in $X$ satisfying
\[ \# \{ 0 \leq i < n \,\mid\, T^ix \in \twr_c\} \geq\tfrac{3}{4}\mu_c n, \hone \text{ for all } n \geq N_0,\]
has measure at least $1 - \frac{\delta_+}{2}$. Denote by $W_1$ the set given by property $P(\alpha, \delta, N)$. We will suppose WLOG that $\delta_+ < 4\delta < \frac{1}{10}.$ Let $N_+$ sufficiently large so that
\[ N_+ > \max\left\{ N, \tfrac{N_0}{\delta_+}\right\}\]
and the set $W$ of points in $X$ satisfying for all $n > \frac{N_+}{2}$
\begin{enumerate}
	\item $ \# \{ 0 \leq i < n \,\mid\, T^ix \in W_0\} \geq (1 - \delta_+)n,$
	\item $ \# \{ 0 \leq i < n \,\mid\, T^ix \in W_1\} \geq (1 - 2\delta)n,$
	\item $ \# \{ 0 \leq i < n \,\mid\, T^ix \in F\} \geq \tfrac{3\mu_\twr}{4h}n,$
\end{enumerate}
has measure at least $1 - \delta_+$. Both $N_0$ and $N_+$ are well defined by the pointwise ergodic Theorem. Fix $x, y \in W,$ $n \geq N_+$ and denote $v = x^{\Pa, n},$ $w = y^{\Pa, n}$. To prove the result it suffices to show 
\[ \overline{f}_n (v, w) < 1 - \alpha_+.\]
For an appropriate choice of $c$ depending only on $\alpha$ (see equation (\ref{cdef})) the Proposition will follow from the following Claims. \\
 
Using the recurrence of $x$ and $y$ to $\twr_c$ one can prove the following: \\

\textbf{Claim 1.} There exists a natural number $L$ and splittings $v = v^1 \cdot \dots \cdot v^{2L + 1}$, $w = w^1 \cdot \dots \cdot w^{2L + 1}$ such that: 
\begin{enumerate}
	\item $\frac{nc}{2}\mu_\twr^2 \leq hL \leq nc\mu_\twr^2$.
	\item $|v^{2k}| = h = |w^{2k}|$ for all $k = 1, \dots, L$. 
	\item $v^{2k}_0, w^{2k}_0 \in F$ for all $k = 1, \dots, L$. 
	\item $\sum_{k = 0}^L \left||v^{2k + 1}| - |w^{2k + 1}|\right| \leq chL + \delta n$. \\
\end{enumerate} 

Using the recurrence of $x$ and $y$ to the set $W_1$ given by $P(\alpha, \delta, N)$ and the fact that most of the odd blocks will be sufficiently big (condition (1) in Claim 1) and have comparable size (condition (4) in Claim 1) one can prove the following:\\

\textbf{Claim 2.} There exist partial couplings between the odd blocks of the splittings of $v, w$ in Claim 1 allowing to match at least
\[ n\alpha\left( 1 - \frac{Lh}{n}\left(1 + \dfrac{c}{2}\right) - 6\delta \right)\]
coordinates of $v$ and $w$. \\

Before proving Claims 1 and 2 let us who how they imply the result. Let
\[ \Qa = \twr \cup \Pa|_{\twr^c}.\]
By Lemma \ref{thm: names_and_towers} we can suppose WLOG that \[d_{n}^\Pa(x, y) < d_{n}^\Qa(x, y) + 3\Delta(\twr, \Pa).\]
Since any two points in the floor set $F$ of the tower define the same $\Qa$-$h$-name and by conditions (2), (3) in Claim 1 we have
\[\sum_{k = 1}^L \#\left\{0 \leq i < h \,\big|\, v_i^{2k} \neq w_i^{2k} \right\} \leq 3\Delta(\twr, \Pa)n.\]
Thus we can match at least 
\[hL - 3\Delta(\twr, \Pa)n \]
coordinates of the even blocks associated to the splittings of $v, w$. Therefore, by Claims 1 and 2 we can match at least 
\[ n\left( \alpha + \frac{hL}{n}\left( 1-\alpha\left(1 + \dfrac{c}{2}\right)\right) - 6\delta - 3\Delta(\twr, \Pa)\right)\]
coordinates of $v$ and $w$. Hence 
\[ \overline{f}_n (v, w) \leq 1 - \alpha - \frac{hL}{n}\left( 1-\alpha\left(1 + \dfrac{c}{2}\right)\right) + 6\delta + 3\Delta(\twr, \Pa).\]
 Therefore for
\begin{equation}
\label{cdef}
c = \min\left\{\dfrac{9}{10}, \dfrac{1 - \alpha}{\alpha}\right\}
\end{equation} 
the result follows. In fact, with this definition of $c$ and by (1) in Claim 1
\[ \frac{hL}{n}\left( 1-\alpha\left(1 + \dfrac{c}{2}\right)\right) \geq \frac{\mu_\twr^2}{10}(1 - \alpha)^2.\]
\end{proof}

\begin{proof}[Proof of the Claim 1] 
Let $m = (1 - 2\delta)n$ and denote $l = \left\lceil \tfrac{3\mu_\twr}{4h}m \right\rceil.$ By properties (1), (3) in the definition of $W$ there exist sequences $0 < t_1 < \dots < t_l \leq m$ and $0 < t_1' < \dots < t_l' \leq n$ obeying:
 \begin{itemize}
 	\item $T^{t_k} \in F$, $T^{t_k'}y \in W_0$, 
	\item $t_k' \geq t_k$,
	\item $t_{k + 1}' - t_k' \geq t_{k + 1} - t_k$,
	\item $\sum_{k = 1}^l (t_k' - t_k) < \delta_+ n.$
 \end{itemize}
Thus for any $1 \leq k \leq l$ and by definition of $W_0$ and $N_0$ 
\[ \#|\{ 0 \leq j \leq N_0 \,\mid \, T^{t_k' + j} y \in \twr_c \} | \geq \dfrac{3 c\mu_\twr}{4}N_0.\] 
Hence there exists $0 \leq j \leq N_{0}$ such that
\begin{equation}
\label{GoodIndex} 
 \#|\{ 0 \leq k \leq l \,\mid \, T^{t_k' + j} y \in \twr_c \} | \geq \dfrac{3c\mu_\twr}{4}l.
 \end{equation}
 Let $L := \left\lceil\tfrac{3 c\mu_\twr}{4}l \right\rceil.$
By definition of $t_k, t_k'$ and by (\ref{GoodIndex}) there exist natural numbers 
\begin{equation}
\label{eq: splitting_sequences}
\begin{aligned} 
0 \leq i_1 < i_1 + h < i_2 < i_2 + h < \dots< i_L < i_L + h < n, \\
0 \leq j_1 < j_1 + h < j_2 < j_2 + h < \dots < j_L < j_L + h < n,
\end{aligned}
\end{equation}
such that $ T^{i_k}x, T^{j_k} y \in F$ for all $1 \leq k \leq L$ and
 \begin{equation*}
 \label{IndexConditions}
\sum_{k = 1}^L |i_k - j_k - j| < Lch + 2\delta_+ n.
 \end{equation*}
 Let us assume by simplicity that $i_1, j_1 > 0$. Then we can use the index sequences in (\ref{eq: splitting_sequences}) to split $v, w$ as a concatenation of non empty words 
 \[v = v^1 \cdot v^2 \cdots v^{2L+1}, \hspace{1cm} w = w^1 \cdot w^2 \cdots w^{2L+1}. \]
such that for all $1 \leq k \leq L$ 
\[v^{2k}_0, w^{2k}_0 \in F, \hone |v^{2k}| = h = |w^{2k}| \]
and
\begin{align*}
 \sum_{k= 0}^{L}\left||v^{2k + 1}| - |w^{2k + 1}|\right| & < chL + 2\delta_+n + 2N_0 \\
 & < chL + \delta n 
 \end{align*}
 \end{proof}

\begin{proof}[Proof of the Claim 2]
Notice that any $n$-word $w$ admits a unique decomposition of the form
\[ w = \prod_{i = 1}^h \sigma^i\]
where each $\sigma^i$ satisfies one of the following:
\begin{enumerate}
\item $\sigma^i_j \notin W_1$ for all $0 \leq j < |\sigma^i|$. 
\item $\sigma^i_0 \in W_1$ and $|\sigma^i| = N$. 
\item $i = h$, $\sigma^h_0 \in W_1$ and $|\sigma^h| < N$. 
\end{enumerate}
Let us denote 
\[ R(\omega) = N \left| \{ 0 < i \leq h \mid \sigma_0^i \in W, \, |\sigma^i| = N \}\right|. \]
Define
\[ R_k = \min\{R(v^k), R(w^k)\}. \]
By definition of $W_1$ for all $0 \leq k \leq L$ we can match at least $\alpha R_k$ coordinates of the blocks $(v^{2k+1}, w^{2k + 1})$. By definition of $W$ at most $2\delta n$ of the first $n$ iterates of $x$ and $y$ do not belong to $W_1$. Recall that the sum of the lengths of the odd blocks associated to each splitting is equal to $n - Lh$. Hence
\begin{align*}
\sum_{k = 0}^{L} R(v^{2k + 1}), \sum_{k = 1}^{L} R(w^{2k + 1}) & \geq n - Lh - 2\delta n - LN 
\end{align*}
Notice that
\begin{align*}
R_k & = \frac{1}{2}\left( R(v^k) + R(w^k) - |R(v^k) - R(w^k)|\right) \\
& \geq R(v^k) + R(w^k) - \frac{1}{2}\left( |v^k| + |w^k| + \left||v^k| -|w^k|\right|\right)
\end{align*}
Thus by (4) in Claim 1 it follows that
\begin{equation*}
\begin{aligned}
\sum_{k = 0}^L R_{2k + 1} & \geq n - Lh - LN - \frac{c}{2}hL - 5\delta n \\
 & \geq n\left(1 - \frac{Lh}{n}\left(1 + \frac{c}{2}\right) - 6\delta \right)
\end{aligned}
\end{equation*}
\end{proof}

\subsection{LB criterion for ergodic zero-entropy transformations}

In this section we combine Propositions \ref{thm: LB_process_criterion} and \ref{thm: LB_prop} to provide a new loose-Bernoullicity criterion for ergodic zero-entropy processes (Item 3 of Proposition \ref{thm: LB_criterion_towers}) which relies only on the existence of an appropriate sequence of Rokhlin towers. As a direct consequence, we obtain a loose-Bernoullicity criterion for ergodic zero-entropy transformations (Theorem \ref{thm: LB_criterion_transformation}). \\

By iterating Proposition \ref{thm: LB_prop} we will show the following.

\begin{prop}
\label{thm: LB_criterion_towers}
Let $(T, \Pa)$ be an ergodic zero-entropy process, $L \in \N \cup \{ +\infty\}$ and $\mathfrak{T} = \{ \twr_n\}_{1 \leq n < L}$ be a collection of Rokhlin towers. Suppose $\twr_{n}$ is $\delta_{n - 1}$-monochromatic with respect to $\Pa$ for all $1 \leq n < L$ where
\[ \alpha_0 = 0, \hone \alpha_{n + 1} = \alpha_{n} + \delta_n, \hone \delta_n = \mu_{\twr_{n + 1}}^2 \left( \frac{1 - \alpha_n}{10} \right)^2.\]
Then there exist natural numbers $N_n(T, \Pa, \twr_1, \dots, \twr_{n})$ for $0 \leq n < L$ such that for any $1 \leq l \leq L$ satisfying
\begin{equation*}
\label{eq: height_hypothesis}
h_{\twr_{n}} > N_{n - 1} \hone \text{ for all } 1 \leq n < l,
\end{equation*} 
the following holds:
\begin{enumerate}
\item There exist $(m_n)_{0 \leq n < l} \subset \N$ such that $(T, \Pa)$ verifies $P(\alpha_n, \delta_n, m_n)$ for all $0 \leq n < l$. 
\item If $l < +\infty$ and there exists a Rokhlin tower $\twr$, $\delta_\twr$-monochromatic with respect to $\Pa$ where
\[\delta_\twr = \mu_{\twr}^2\left(\frac{1 - \alpha_{l - 1}}{10}\right)^2\]
obeying 
\[h_\twr > N_l\]
 then the process $(T, \Pa)$ verifies $P(\alpha_{l - 1} + \delta_\twr, \delta_\twr, m_l)$ for some $m_l \in \N$. 
 \item If $l = +\infty$ and 
 \begin{equation}
 \label{eq: measure_condition}
 \sum_{i = 1}^{+\infty} \mu_{\twr_n}^2 = +\infty
 \end{equation}
 then the process $(T, \Pa)$ is loosely Bernoulli. 
\end{enumerate}
\end{prop}

\begin{proof} By the pointwise ergodic Theorem there exists $m_0$ depending only on $T$ and $\Pa$ so that property $P(\alpha_0, \delta_0, m_0)$ is verified. In fact, given $P \in \Pa$ with positive measure it suffices to take $m_0$ sufficiently large so that the set of points $x$ in $X$ for which at least one of the iterates $T^i x$ with $0 \leq i < m_0$ belongs to $P$ has measure bigger than $1 - \delta_0$. Let
 \[ N_0 = \left\lceil \dfrac{2m_0}{\delta_0} \right\rceil.\]
 We will define recursively sequences of natural numbers $m_n,$ $N_n = \left\lceil \frac{2m_n}{\delta_n} \right\rceil$ as follows. Let $1 \leq n < L$ and suppose $m_0, \dots, m_{n - 1}$ (and thus $N_0, \dots, N_{n - 1}$) have been defined. If
 \[ h_{\twr_n} < N_{n - 1} \]
 we define $N_k = N_{n - 1}$, $m_k = m_{n - 1}$ for all $n \leq l < L$ and the recursive process ends. Otherwise, if 
\[ h_{\twr_n} \geq N_{n - 1} \]
we can apply Proposition \ref{thm: LB_prop} with $\alpha = \alpha_{n - 1},$ $\delta = \delta_{n}$, $N = {N_{n - 1}}{}$, $\twr = \twr_{n},$ $\delta_+ = \delta_n$ and define
 \[ m_{n} = N_+(\alpha_{n - 1}, \delta_{n - 1}, N_{n - 1}, \twr_n, \delta_n).\]
 Notice that in this case, for $\alpha = \alpha_{n - 1},$ $\delta = \delta_{n - 1}$ and $\twr = \twr_{n}$ in equation (\ref{eq: new_alpha}) 
 \begin{equation*}
 \label{eq: new_alpha_inequality}
 \alpha_+ \geq \alpha_{n - 1} + \mu_{\twr_n}^2 (1 - \alpha_{n - 1})^2\left( \frac{1}{10} - \frac{9}{100} \right) = \alpha_{n}
 \end{equation*} 
and thus property $P(\alpha_n, \delta_n, m_n)$ is satisfied. By induction, there exist sequences of natural numbers $(N_n)_{0 \leq n < L}$ and $(m_n)_{0 \leq n < L}$ for which the first assertion holds. The second assertion now follows from Proposition \ref{thm: LB_prop}. Finally, if $l = +\infty$ it follows from (\ref{eq: measure_condition}) that $\alpha_n \nearrow 1$ and $\delta_n \searrow 0$. Therefore, the process $(T, \Pa)$ is loosely Bernoulli by Proposition \ref{thm: LB_process_criterion}.
\end{proof}

By applying item 3 in Proposition \ref{thm: LB_criterion_towers} to a sequence of refining measurable partitions we immediately obtain the following loose-Bernoullicity criterion.

\begin{thm}
\label{thm: LB_criterion_transformation}
Let $(X, T)$ be an ergodic zero-entropy transformation and let $\{ \Pa_n\}_{n \in \N} \subset \FP$ be a sequence of refining partitions satisfying $\Pa_n \rightarrow \mathcal{E}$. Suppose there exists a collection of Rokhlin towers $\mathfrak{T} = \{ \twr_n\}_{n \geq 1} \subset \TSet$ obeying:
\begin{enumerate}[(i)]
 \item \label{measure_condition} $\sum_{n = 1}^{+\infty} \mu_{\twr_n}^2 = +\infty,$
 \item $\twr_{n + 1}$ is $\delta_{i, n - i}$-monochromatic with respect to $\Pa_i$ for all $0 \leq i \leq n + 1$, where 
 \[ \alpha_{i, 0} = 0, \hone \alpha_{i, n + 1} = \alpha_{i, n} + \delta_{i, n}, \hone \delta_{i, n} = \mu_{\twr_{n + 1 + i}^2} \left( \frac{1 - \alpha_{i, n}}{10} \right)^2.\]
\end{enumerate}
There exist natural numbers $M_n(T, \Pa_1, \dots, \Pa_n, \twr_1, \dots, \twr_{n})$ such that if 
\[ h_{\twr_{n}} > M_{n - 1} \]
holds for all $n \geq 1$, then the transformation $T$ is loosely Bernoulli. 
\end{thm}

\noindent \textbf{Remark:} $M_n$ in the previous Theorem can be taken as
\[ M_n = \max_{0 \leq i \leq n} N_i(T, \Pa_i, \twr_i, \dots, \twr_n)\]
with $N_i$ as defined in Proposition \ref{thm: LB_criterion_towers}. \\

Let us point out that the hypotheses of Theorem \ref{thm: LB_criterion_transformation} are always satisfied by local rank one systems and thus Theorem \ref{thm: LB_criterion_transformation} can be seen as a generalization, in the zero-entropy case, of Ferenczi's result in \cite[Proposition 1]{ferenczi_systemes_1984}, which shows that ergodic local rank one transformations are loosely Bernoulli. \\

Recall that a system $\DSS$ is said to have \textit{local rank one} if there exists $a > 0$ such that for every finite measurable partition $\Pa \in \FP$ and for every $\epsilon > 0$ there exists a Rokhlin tower $\twr \in \TSet$, with $\mu_\twr > a$, and a measurable partition of the support of the tower $\Qa$, obeying $\twr \preccurlyeq \Qa$, such that
\[ D(\Qa, \Pa|_\twr) < \epsilon. \]
In a few words, local rank one property guarantees that any finite measurable partition can be arbitrarily well approximated by towers whose size is bounded from below by a fixed positive constant. \\ 

We stress the fact that the size of the towers in Theorem \ref{thm: LB_criterion_transformation} is not necessarily bounded from below. 
 
\subsection{LB criterion for the cartesian product} 

In this section we adapt Proposition \ref{thm: LB_criterion_towers} and Theorem \ref{thm: LB_criterion_transformation} to obtain similar loose-Bernoullicity criteria, compatible with mixing, for the cartesian product of ergodic zero-entropy processes (Item 3 in Proposition \ref{thm: LB_criterion_towers_product}) and transformations (Theorem \ref{thm: LB_criterion_transformation_product}).\\

Notice that any pair of Rokhlin towers $\twr^{\pm} = \twr(F^{\pm}, h^{\pm})$ induces a tower $\twr = \twr(F^+\times F^-, \max\{h^+, h^-\})$ for the cartesian product $T \times T$ whose size is given by
\[\mu_\twr = \frac{\mu_{\twr^+} \mu_{\twr^-}}{\min\{h^+, h^-\}}.\]

The height of this product tower can be increased, at the cost of slightly shrinking the base, if some almost periodicity holds for the initial towers, that is, if most of their roof's image is contained in their base. To illustrate this let us suppose for a moment that 
\begin{equation}
\label{eq: periodic_tower}
T^h(F^{\pm}) = F^{\pm}.
\end{equation} 
Then if $h^+, h^-$ are relatively prime $\twr(F^+ \times F^-, h^+ h^-)$ is a well defined Rokhlin tower of size $\mu_{\twr^+} \mu_{\twr^-}$. Clearly, if the system is mixing equation (\ref{eq: periodic_tower}) cannot hold unless $F^{\pm} = X$. Nevertheless, this argument can be adapted provided the \textit{precision} of the towers $\twr^{\pm}$ is sufficiently small. Recall that the precision of a Rokhlin tower $\twrb$, which we denote by $\rho_\twr$, is given by
\[ \rho_\twr = \mu (F \Delta T^h(F)). \]
The following Lemma gives explicit relations between the towers $\twr^{\pm}$ and its cartesian product. 

\begin{lem}
\label{thm: product_tower}
Let $\DS$ be a finite measure preserving transformation, $\Pa^{\pm} \in \FP$ and $\twr^{\pm} = \twr(F^{\pm}, h^{\pm}) \in \TSet$ Rokhlin towers of sizes $\mu^{\pm}$ with $h_+, h_-$ relatively prime. If
\begin{equation}
\label{PrecisionCondition}
 (h^{\mp} - 1) \rho_{\twr^\pm} < (1 - c)\dfrac{\mu^{\pm}}{h^{\pm}},
\end{equation}
for some $0 < c < 1$, then there exists $F \subset F^+ \times F^-$ such that $\twr(F, h^+h^-)\in \twr(X \times X, T \times T)$ is a well-defined Rokhlin tower of size $\mu_{\twr} > c^2\mu_+ \mu_-{}$ satisfying
\[ d(\twr, \Pa^+ \times \Pa^-) \leq d(\twr^+, \Pa^+)\mu_{\twr^-} + d(\twr^-, \Pa^-)\mu_{\twr^+}.\]
\end{lem}

\begin{proof}
 Let 
 \[E^{\pm} = \bigcap_{k = 1}^{h^{\mp} - 1} T^{-kh^{\pm}}(F^{\pm}).\]
By (\ref{PrecisionCondition}) 
\begin{equation}
\label{MeasureProduct}
\mu(E^{\pm}) > \frac{c\mu^{\pm}}{h^{\pm}}.
\end{equation}
Let $F = E^+ \times E^-$. Notice that for any $0 < i < h^+h^-$
\[ F \cap (T \times T) ^i(F) \neq \emptyset \hspace{0.5cm} \text{implies} \hspace{0.5cm} h^{\pm} \, | \, i .\] 
Since $h^+, h^-$ are relatively prime this shows that $\twr(F, h^+h^-)$ is a well defined Rokhlin tower. The estimate on the size of this tower follows trivially from (\ref{MeasureProduct}). Denote $\Pa^{\pm} = \{ P_1^{\pm}, \dots, P_N^{\pm}\}$ and let $ = \{\xi_1^{\pm}, \dots, \xi_N^{\pm}\}$ be partitions of $\twr^{\pm}$ such that 
\[ \twr^{\pm} \preccurlyeq \xi^{\pm}, \hspace{1cm} d(\twr^\pm, \Pa^\pm) = d(\xi^\pm, \Pa^\pm).\]
Denote $\Pa = \Pa^+ \times \Pa^-$. Notice that $\twr = \twr(T \times T, F, h^+h^-)\subset \twr^+ \times \twr^-$ and since $\xi^+\times \xi^-$ is a partition of $\twr^+ \times \twr^-$ then $\xi = (\xi^+\times \xi^-)|_\twr$ is a well defined partition of $\twr$. Furthermore $\twr \preccurlyeq \xi$ and 
\begin{align*}
\Delta(\twr, \Pa) & \leq d(\xi, \Pa|_\twr) \\
& = \sum_{i, j = 1}^N \mu((\xi^+_i \times \xi^-_j \cap \twr) \,\setminus \, (P^+_i \times P^-_j)) \\
& \leq \sum_{i, j = 1}^N \mu(\xi_i^+ \,\setminus\, P_i^+ \times \xi_j^-) + \mu(\xi_i^+ \times \xi_j^- \,\setminus\, P^-_j) \\
& \leq d(\xi^+, \Pa^+)\mu_{\twr^-} + d(\xi^-, \Pa^-)\mu_{\twr^+}
\end{align*}
\end{proof}

Taking into account the previous Lemma the proof of the following result follows the same lines than that of Proposition \ref{thm: LB_criterion_towers}.

\begin{prop}
\label{thm: LB_criterion_towers_product}
Let $(T, \Pa)$ be a weak mixing zero-entropy process, $L \in \N \cup \{ +\infty\}$ and $\mathfrak{T} = \{ \twr_n^\pm\}_{1 \leq n < L}$ be a collection of Rokhlin towers. Suppose $h_{\twr^+_n}, h_{\twr^-_n}$ are relatively prime and
\[ \rho_{\twr_n^{\pm}} \leq \frac{\mu_{\twr_n}^\pm}{2h_{\twr_n}^+h_{\twr_n}^-}\]
for all $1 \leq n < L$. Further assume that $\twr_{n+ 1}^\pm$ are $\delta_{n}$-monochromatic with respect to $\Pa$ for all $1 \leq n < L$ where
\[ \alpha_0 = 0, \hone \alpha_{n + 1} = \alpha_{n} + \delta_n, \hone \delta_n = \left(\frac{\mu_{\twr_{n + 1}^+}\mu_{\twr_{n + 1}^-}}{2}\right)^2\left( \frac{1 - \alpha_n}{10} \right)^2.\]
Then there exist natural numbers $N_n(T, \Pa, \twr_1^\pm, \dots, \twr_{n}^\pm)$ for $0 \leq n < L$ such that for any $1 \leq l \leq L$ satisfying
\begin{equation}
\label{eq: height_hypothesis_product}
h_{\twr_{n}^+}h_{\twr_{n}^-} > N_{n - 1} \hone \text{ for all } 1 \leq n < l
\end{equation}
the following holds:
\begin{enumerate}
\item There exist $(m_n)_{0 \leq n < l} \subset \N$ such that the process $(T \times T, \Pa \times \Pa)$ verifies $P(\alpha_n, \delta_n, m_n)$ for all $0 \leq n < l$. 
\item If $l < +\infty$ and there exist Rokhlin towers $\twr^\pm$, $\delta_\twr$-monochromatic with respect to $\Pa$ where
\[\delta_\twr = \left(\frac{\mu_{\twr^+}\mu_{\twr^-}}{2}\right)^2\left(\frac{1 - \alpha_{l - 1}}{10}\right)^2\]
obeying 
\[h_{\twr^+}h_{\twr^-} > N_l\]
 then the process $(T \times T, \Pa \times \Pa)$ verifies $P(\alpha_{l - 1} + \delta_\twr, \delta_\twr, m_l)$ for some $m_l \in \N$. 
 \item If $l = +\infty$ and 
 \begin{equation}
 \label{eq: measure_condition_product}
\sum_{n = 1}^{+\infty} \Big(\mu_{\twr_{n + 1}^+}\mu_{\twr_{n + 1}^-}\Big)^2 = +\infty,
 \end{equation}
 then the process $(T \times T, \Pa \times \Pa)$ is loosely Bernoulli. 
\end{enumerate}
\end{prop}

In the previous Proposition we assume the system to be weak mixing, instead of ergodic, just to guarantee that its cartesian product is ergodic. Proposition \ref{thm: LB_criterion_towers_product} immediately implies the following. 

\begin{thm}
\label{thm: LB_criterion_transformation_product}
Let $(X, T)$ be a weak mixing zero-entropy transformation and $\{ \Pa_n\}_{n \in \N} \subset \FP$ be a sequence of refining partitions satisfying $\Pa_n \rightarrow \mathcal{E}$. Suppose there exists a collection of Rokhlin towers $\mathfrak{T} = \{ \twr_n^\pm\}_{n \geq 1} \subset \TSet$ obeying:
\begin{enumerate}[(i)]
 \item \label{measure_condition} $\sum_{n = 1}^{+\infty} \big(\mu_{\twr_n^+}\mu_{\twr_n^-}\big)^2 = +\infty,$
 \item $\rho_{\twr_n^{\pm}} \leq \frac{\mu_{\twr_n}^\pm}{2h_{\twr_n}^+h_{\twr_n}^-}$ for all $n \geq 1$,
 \item $h_{\twr^+_n}, h_{\twr^-_n}$ are relatively prime for all $n \geq 1$,
 \item $\twr_{n + 1}^\pm$ is $\delta_{i, n - i}^\pm$-monochromatic with respect to $\Pa_i$ for all $0 \leq i \leq n + 1$, where 
 \[ \alpha_{i, 0} = 0, \hone \alpha_{i, n + 1} = \alpha_{i, n} + \delta_{i, n},\]
 \[ \delta_{i, n} = \left(\frac{\mu_{\twr_{n + 1 + i}^+}\mu_{\twr_{n + 1 + i}^-}}{2}\right)^2\left( \frac{1 - \alpha_{i, n}}{10} \right)^2.\]
\end{enumerate}
There exist natural numbers $M_n(T, \Pa_1, \dots, \Pa_n, \twr_1^\pm, \dots, \twr_{n}^\pm)$ such that if 
\[h_{\twr_{n}^+}h_{\twr_{n}^-} > M_{n - 1} \]
holds for all $n \geq 1$, then the transformation $T \times T$ is loosely Bernoulli. 
\end{thm}

\noindent \textbf{Remark:} $M_n$ in the previous Theorem can be taken as
\[ M_n = \max_{0 \leq i \leq n} N_i(T, \Pa_i, \twr_i^\pm, \dots, \twr_n^\pm)\]
with $N_i$ as defined in Proposition \ref{thm: LB_criterion_towers_product}. \\

A similar criterion for loose-Bernoullicity of the cartesian product of a measure preserving transformation with itself was proven by Katok and Stepin in \cite{katok_metric_1970} assuming the existence of \textit{excellent approximations of type $(n, n + 1)$}. Precise definitions and a proof of this criterion can be found in \cite[Section 3.2]{katok_combinatorial_2003}. We stress the fact that transformations within the hypotheses of Katok and Stepin's criterion cannot be mixing (see \cite{katok_combinatorial_2003}). Also, for a given finite partition $\Pa$, systems with excellent approximations of type $(n, n + 1)$ admit couples of arbitrarily high towers whose heights differ by $1$, whose size is bounded from below and which approximate $\Pa$ arbitrarily well. Therefore, we can apply Theorem \ref{thm: LB_criterion_transformation_product} to recover Katok and Stepin's criterion. 

\section{Smooth mixing transformations with LB cartesian product}
\label{sc: LB_construction}

In this section we will apply the loose-Bernoullicity criterion in Theorem \ref{thm: LB_criterion_transformation_product} to show the existence of zero-entropy smooth mixing loosely Bernoulli transformations on $\T^3$ whose cartesian product with themselves is loosely Bernoulli. These systems will be obtained as the limits of the recursive procedure described in Section \ref{sc: construction} and which we summarize in Lemma \ref{thm: recursive_construction}. \\
 
The transformations we consider are time one maps of continuous special flows over minimal translations of $\T^2$. As we shall see in the following, these maps can be naturally identified with continuous homeomorphisms of $\T^3$ and they possess a natural invariant ergodic measure. Under some additional assumptions these maps will be actually mixing with respect to this measure. Furthermore, by Abramov's \cite{abramov_entropy_1959} entropy formula (\ref{eq: abramov_formula}) these maps will have zero entropy. 

\subsection{Special Flows} 

Given a measure preserving transformation $\DS$ and an integrable function $\varphi : X \rightarrow (0, +\infty)$ denote by $\sim_{T, \varphi}$ the equivalence relation on $X \times \R$ given by
\[ (x, s + \varphi(x)) \sim_{T, \varphi} \left(T(x), s \right).\]
The \textit{special flow over $T$ with roof function $\varphi$} which we denote 
$\Psi^t_{T, \varphi}: X_\varphi \rightarrow X_\varphi,$
 where $X_\varphi = X \times \R /_{\sim_{T, \varphi}}$, is defined as the identification by $\sim_{T, \varphi}$ of the action\[ (t, x, s) \mapsto (x, s + t).\]
Using the Birkhoff sums of $\varphi$ with respect to $T$ and noticing that 
\[ X_\varphi = \left\{ (x, s) \in X \times \R \mid 0 \leq s \leq \varphi(x)\right\} /_{\sim_{T, \varphi}},\]
we can provide a more explicit expression of the special flow. Given $(x, s) \in X \times \R$ obeying $0 \leq s < \varphi(x)$ and for any $t \geq 0$ we have
\[ \Psi^t_{T, \varphi}(x, s) = \Bigg(T^{\pi(t, x, s)}(x), t + s - \sum_{k = 0}^{\pi(t, x, s) - 1} \varphi(T^k)(x,y)\Bigg),\]
where
\[ \pi(t, x, s) = \min\Bigg\{ n \in \N \,\Bigg|\, t + s < \sum_{k = 0}^{n} \varphi(T^k)(x) \Bigg\}.\]
 The special flow $\Psi^t_{T, \varphi}$ preserves the normalized product measure 
\[ \lambda_{T, \varphi} = \frac{(\mu \times \lambda)\mid_{X_\varphi}}{\int_X \varphi d\mu}, \]
where $\lambda$ denotes the Lebesgue measure on $\R$. We recall Abramov's \cite{abramov_entropy_1959} entropy formula
\begin{equation}
\label{eq: abramov_formula}
 h_{\lambda_{T, \varphi}}(\Psi^t_{T, \varphi}) = t \frac{h_{\mu}(T)}{\int_X \varphi d\mu}.
 \end{equation}
 
 \subsubsection{Time one maps of special flows over $\T^2$}

Let us recall that the toral translation $R_\omega: \T^2 \rightarrow \T^2$, where $\omega \in \R^2$ denotes the \textit{translation vector}, is minimal if and only if $(1, \omega)$ is \textit{non-resonant}. A vector $v \in \R^d$ is said to be non-resonant if its coordinates are \textit{rationally independent}, i.e. if the equation $\langle v, k \rangle = 0$ does not admit a solution $k \in \Z^d \setminus \{ 0 \},$ otherwise it is called \textit{resonant}. \\

Given $r \in \N \cup \{+\infty\}$ let $C^r_+(\T^2) $ denote the space of strictly positive $C^r$ functions on $\T^2$ endowed with the $C^r$ topology. We denote $C^0_+(\T^2)$ simply by $C_+(\T^2)$. To simplify the notations, for $\omega \in \R^2$ fixed we denote 
\[ \tmap{\varphi} = \Psi^1_{R_{\omega}, \varphi}, \hone \lambda_{\omega, \varphi} = \lambda_{R_{\omega}, \varphi}\]
 for every $\varphi \in C_+(\T^2)$. By (\ref{eq: abramov_formula}) 
 \[h_{\lambda_{\omega, \varphi}}(\tmap{\varphi}) = 0.\]
 Notice that for $\varphi, \psi \in C_+(\T^2)$ the transformations $\tmap{\varphi}$, $\tmap{\psi}$ are not necessarily defined on the same space. To be able to compare these two mappings we first normalize them as follows. Let $\omega = (\Omega, \Omega') \in \R^2$. For every $\varphi \in C_+(\T^2)$ let
\[ \begin{array}{cccc}
 \Phi_{\omega, \varphi}: & \T^3 & \longrightarrow & X_\varphi \\
 & (x, y, z) & \mapsto & (x + \Omega, y + \Omega', z\varphi(x, y)) 
 \end{array}.
 \]
It is easy to show that $\Phi_{\omega, \varphi}$ is a well defined homeomorphism. Define
 \begin{equation}
 \label{eq: normalization}
 G_{\omega, \varphi} = \Phi_\varphi^{-1} \circ \tmap{\varphi} \circ \Phi_\varphi, \hone \mu_{\omega, \varphi} = \Phi_\varphi^*(\lambda_{\omega, \varphi}).
 \end{equation}
Thus $\transf{\varphi}$ is a well defined measure preserving homeomorphism of $\T^3$ for any $\varphi \in C_+(\T^2)$. Furthermore, the transformation $\Phi_{\omega, \varphi}$ (and therefore $\map{\varphi}$ and $\meas{\varphi}$) is as regular as the roof function $\varphi$. \subsection{Mixing criterion}

Given $\omega = (\Omega, \Omega') \in \R^2$ such that $(1, \omega)$ is non-resonant we denote by $q_n$, $q_n'$ the \textit{first return times} of $\Omega$ and $\Omega'$ respectively. Recall that given $\alpha \in (0,1)$ irrational and by setting $q_0 = 1$ the first return times of $\alpha$ can be defined recursively for all $n \in \N$ by
\[q_{n + 1} = \min \left\{ k \in \N^* \mid \vvvert k\alpha \vvvert < \vvvert q_n\alpha \vvvert \right\}, \]
where $\vvvert \cdot \vvvert$ denotes the distance to the closest integer. We extend the definition of $\vvvert \cdot \vvvert$ to vectors in $\R^d$ by
\[ \vvvert x \vvvert = \min_{k \in \Z^d}|k - x|_1.\]
For $\alpha \in \R$ let us denote
\[ \{ \alpha\} = \alpha - \lfloor \alpha \rfloor. \]
The following mixing criterion was proved by Fayad in \cite[Proposition 3.3]{fayad_smooth_2006}.

\begin{prop}
\label{thm: mixing_criterion}
Suppose
\begin{equation}
\label{eq: rotation_conditions}
q_n' \geq e^{3q_n}, \hspace{1cm} q_{n+1}, \geq e^{3q_n'} \hone \text{ for all } n \in \N. 
\end{equation}
If for every $n \in \N$ sufficiently large there exist two sets $I_n, I_n'$, each one being equal to the circle minus two intervals whose lengths converge to zero, and if 
\begin{itemize}
\item For any $y \in \T$, any $x$ such that $\{q_n x\} \in I_n$, and any $m \in [e^{2q_n}/2, 2e^{2q_n'}]$, we have
\[ |\partial_xS_m\varphi(x,y)| \geq \dfrac{m}{e^{q_n}}\dfrac{q_n}{n}\]
\item For any $x \in \T$, any $y$ such that $\{q_n y\} \in I_n'$, and any $m \in [e^{2q_n}/2, 2e^{2q_{n+1}}]$, we have
\[ |\partial_yS_m\varphi(x,y)| \geq \dfrac{m}{e^{q_n'}}\dfrac{q_n'}{n}\]
\end{itemize}
then the special flow $\Psi^t_{R_\omega, \varphi}$ is mixing. 
\end{prop}

Let us point out that the set of vectors verifying (\ref{eq: rotation_conditions}) is uncountable and dense. For a proof of this see \cite[Appendix 1]{yoccoz_centralisateurs_1995}. 

\subsection{Construction} 
\label{sc: construction}
In the following we fix a vector $\omega = (\Omega, \Omega') \in \R^2$ such that $(1, \omega)$ is non-resonant and denote by $q_n$, $q_n'$ the first return times of $\Omega$ and $\Omega'$ respectively. We further assume that (\ref{eq: rotation_conditions}) holds. Let
\[ \MixColl = \left\{ \varphi \in C^\infty_+(\T^2) \mid \transf{\varphi} \text{ is mixing}\right\}.\]
Denote by $\varphi_0 : \T^2 \rightarrow \R$ the real analytic function given by 
\begin{equation}
\label{eq: varphi_definition}
 \varphi_0(x, y) = 1 + \sum_{n = 1}^{+\infty} X_n(x) + Y_n(y)
 \end{equation}
where
\begin{equation}
\label{eq: X_polynomials}
 X_n(x) = e^{-q_n}\cos(2\pi q_n x), \hone Y_n(y) = e^{-q_n'}\cos(2\pi q_n' y).
 \end{equation}
Applying the mixing criterion in Proposition \ref{thm: mixing_criterion} one can prove the following (see \cite[Theorem 1]{fayad_analytic_2002}).
\begin{prop}
\label{thm: mixing_density}
For any trigonometric polynomial $P : \T^2 \rightarrow \R$ such that $\varphi_0 + P$ is strictly positive \[\varphi_0 + P \in \MixColl\] In particular $\MixColl$ is $C^{\infty}$-dense in $C^\infty_+(\T^2)$.
 \end{prop} 

Our goal is to prove the following.

\begin{thm}
\label{thm: mixing_LB_existence}
There exist a set $S \subset C^\infty_{+}(\T^2)$, $C^\infty$-dense in 
\[\left\{ \varphi \in C^\infty_{+}(\T^2) \,\,\bigg|\, \int_{\T^2} \varphi(\theta) d\theta = 1 \right\}\]
such that for every $\varphi \in S$ the transformation $\transf{\varphi}$ given by (\ref{eq: normalization}) is mixing and its cartesian product with itself $\cartesian{\varphi}$ is loosely Bernoulli. 
\end{thm}

Let us give an outline of the proof. We prove Theorem \ref{thm: mixing_LB_existence} by showing that for any trigonometric polynomial $P$ of zero average such that $\varphi_0 + P$ is strictly positive the following holds: For all $r \in \N$ and all $\epsilon > 0$ there exist an increasing sequence of natural numbers $m_n$ and a sequence of trigonometric polynomials $P_{m_n}: \T \rightarrow \R$ of zero average and degree at most $q_{m_n}$ obeying
\[ \sum_{n = 1}^{+\infty} \|P_{m_n}\|_{C^r} < \epsilon, \]
for which, redefining $\varphi_0$ as
\[ \overline{\varphi}_0 = \sum_{k = 1}^{+\infty} \overline{X}_k(x) + Y_k(y)\]
where $\overline{X}_k$ is equal to $P_{m_n}$ if $k = m_n$ for some $n \geq 1$ or to $X_k$ otherwise, 
\[\psi = \overline{\varphi}_0 + P\]
is a well defined function in $C^\infty_+(\T^2)$ such that $\transf{\psi}$ is mixing and its cartesian product $\cartesian{\psi}$ is loosely Bernoulli. Notice that up to take $m_1$ sufficiently large we can suppose WLOG that
\[ \| \psi - (\varphi_0 + P) \|_{C^r}= \left\| \sum_{n = 1}^{+\infty} P_{m_n} - X_{m_n} \right\|_{C^r} < 2\epsilon.\]

 The sequences $\{m_n\}_{n \geq 1}$ and $\{P_{m_n}\}_{n \geq 1}$ are defined recursively with respect to a sequence of refining partitions $\Pa_n \rightarrow \mathcal{E}$ set beforehand. We define $m_n,$ $P_{m_n}$ by guaranteeing (Proposition \ref{thm: polynomial_perturbation}), at each step, the existence of Rokhlin towers $\twr(\map{\psi_n}, F_n^\pm, h_n^\pm)$ within the hypotheses of Lemma \ref{thm: product_tower} for the processes $(\map{\psi_n}, \Pa_i)$, for $ i = 0, \dots, n$, where
\[ \psi_n = \varphi_0 + P + \sum_{k = 1}^n (P_{m_k}(x) - X_{m_k}(x)).\]
The towers $\twr(\map{\psi_n}, F_n^\pm, h_n^\pm)$ will be sufficiently monochromatic with respect to the partitions $\Pa_0, \dots, \Pa_n$, and of size $\sim (n + 1)^{-\frac{1}{4}}$. By ensuring that $\psi_{n + 1}$ is sufficiently close to $\psi_n$ at each step, we can also guarantee that $\twr(\map{\psi_m}, F_n^\pm, h_n^\pm)$ are well defined Rokhlin towers for all $m \geq n$. \\

By carefully choosing the sequence $\{ h_n^\pm\}_{n \geq 1},$ we can guarantee that for all $1 \leq n \leq m$ the process $(\map{\psi_m}, \Pa_n)$ and the towers $\{ \twr(\map{\psi_m}, F_i^\pm, h_i^\pm)\}_{1 \leq i \leq n}$ verify (\ref{eq: height_hypothesis_product}) with $l = n - 1$. Taking
\[ \psi = \lim_{n \rightarrow +\infty} \psi_n,\]
the loose-Bernoullicity of $\cartesian{\psi}$ will follow from Theorem \ref{thm: LB_criterion_transformation_product}. By Proposition \ref{thm: mixing_criterion}, the transformation $\transf{\psi}$ will be mixing. \\

 A precise statement of the recursive procedure described above is given in Lemma \ref{thm: recursive_construction}.

\subsection{Proof of Theorem \ref{thm: mixing_LB_existence}}

Let us start by noticing that, under some particular conditions on $\Pa$, the property $P(\Pa, \alpha, \delta, n)$ is open in the space of transformations 
\[\left\{ \transf{\varphi}\right\}_{\varphi \in C_+(\T^2)} \subset \text{Hom}(\T^3)\] endowed with the $C^0$ topology. 
 \begin{prop}
\label{thm: open_property}
Let $\alpha, \delta > 0$, $n \in \N$ and suppose $\Pa$ is a finite partition of $\T^3$ whose elements are either open or have zero Lebesgue measure. Then the set of $\varphi \in C_+(\T^2)$ for which $\transf{\varphi}$ verifies $P(\Pa, \alpha, \delta, n)$ is open in $C_+(\T^2)$.
\end{prop}
\begin{proof}
Let us assume $\delta < 1$ otherwise the result is trivial. Let $\varphi \in C_+(\T^2)$ such that $\transf{\varphi}$ verifies $P(\Pa, \alpha, \delta, n)$ and denote by $W_0$ the set associated to this property. Since we consider $\T^3$ endowed with the Borel $\sigma$-algebra we can suppose WLOG, up to slightly reduce the measure of $W_0$, that $W_0$ is a compact set. Furthermore, we can assume that all the iterates of points in $W_0$ belong to open sets of the partition since those who do not verify this condition form a set of zero Lebesgue measure (recall that $\meas{\varphi}$ is equivalent to the Lebesgue measure on $\T^3$). By continuity there exists $\epsilon > 0$ such that for any $\psi \in C_+(\T^2)$ obeying 
\[\|\psi - \varphi \|_{C^0} < \epsilon\] 
 we have $\meas{\psi}(W) > 1 - \delta$ and for all $x \in W$ the $\Pa$-$n$-names associated to $x$ with respect to $\transf{\varphi}$ and $\transf{\psi}$ coincide. Therefore $\transf{\psi}$ verifies $P(\Pa, \alpha, \delta, n)$.
\end{proof}

Let us introduce the family of polynomials to be used in the recursive procedure described in Section \ref{sc: construction}. The class we define here is a slight modification of the one considered in \cite[Proposition 3.5]{fayad_smooth_2006} by Fayad.

\begin{prop}
\label{thm: prop_polynomials}
Let $r\in \N$ and $0 < \mu < \frac{1}{4}$. For all $n \in \N$ sufficiently large there exists
\[ \frac{q_ne^{-q_n}}{2} \leq \eta_n \leq \frac{3q_ne^{-q_n}}{4}\]
and a trigonometric polynomial $P_{\mu, n}$ of zero average and degree at most $q_{n+1}$ satisfying the following:
\begin{enumerate}
 \item $\frac{1}{\eta_n} \in 2q_n\N$,
 \item $\| P_{\mu, n}\|_{C^r} \leq e^{-3q_n/4}$,
 \item $|P_{\mu, n}(x) \mp \eta_n| \leq \frac{e^{3q_n/4}}{q_{n+1}}$ for $|\{q_nx\} -\frac{1}{2} \pm \frac{1}{4}| < 3\mu,$
 \item $\mp P_{\mu, n}'(x) \geq \frac{q_n^2}{e^{q_n}}$ for $|\{q_nx\} - \frac{1}{2} \pm \frac{1}{4}| > 4\mu.$
\end{enumerate}
Furthermore, for any increasing sequence $\{m_n\}_{n \in \N} \subset \N$ and any decreasing sequence $\mu_n \searrow 0$ such that
\[\psi = \varphi_0 + P + \sum_{n = 1}^{+\infty} \left( P_{\mu_n, m_n}(x) - X_{m_n}(x) \right)\]
is a well defined function in $C^\infty_+(\T^2)$ the transformation $\transf{\psi}$ is mixing. 
\end{prop}

The proof of this Proposition is completely analogous to that of Proposition 3.5 and Theorem 3.1 in \cite{fayad_smooth_2006} so we defer it to the Appendix. Using the family of polynomials just defined we will prove the following. 

\begin{prop}
\label{thm: polynomial_perturbation}
Let $P$ be a trigonometric polynomial with zero average such that $\varphi_0 + P$ is strictly positive. Suppose $\Pa$ is a finite partition of $\T^3$ whose elements are either open or have zero Lebesgue measure. For all $r, n, N \in \N$ and all $0 < \beta, \mu, \epsilon < \frac{1}{10}$ there exists $m \in \N$ obeying 
\[ |q_m^r e^{-q_m}| < \epsilon \]
such that for $P_{\mu, m}$ as defined in Lemma \ref{thm: prop_polynomials}
\[\varphi(x, y) = \varphi_0(x, y) + P(x, y) + P_{\mu, m}(x) - X_m(x) \]
is a well defined function in $C^\infty_+(\T^2)$ and the transformation $\transf{\varphi}$ admits two Rokhlin towers $\twr^\pm$ with closed sets as floor sets obeying
\[ \Delta(\twr^{\pm}, \Pa_{n}) \leq \beta, \hspace{0.8cm} \mu_{\twr^\pm} > \mu, \hspace{0.8cm} h_{\twr^-} = h_{\twr^+} - 2 > N, \hspace{0.8cm} \rho_{\twr^{\pm}} < \frac{\mu}{2(h_{\twr^+})^2}. \]
\end{prop}

\begin{proof}
 Let $m \in \N$ and define
\[ L_m = \sum_{n = 1}^{m - 1} X_n(x) + Y_n(y) , \hone H_m = Y_m(y) + \sum_{n > m} X_n(x) + Y_n(y). \]
Thus we can express $\varphi$ in the statement as 
\begin{equation}
\label{Decomposition}
\varphi = 1 + P + L_m + H_m + P_{\mu, m}.
\end{equation}
Take $m$ sufficiently large so that 
\[\deg(P) \leq q_{m-2}, \hone q_m^r e^{-q_m} < \epsilon, \hone \frac{e^{q_m}}{q_m^2} > N.\]
Let 
\[I^\pm = \left\{ x \in \T \, \left| \, \big|x - \tfrac{1}{2q_m} \pm \tfrac{1}{4q_m}\big|\right. \leq \tfrac{2\mu}{q_m}\right\}\]
 and define 
\[ F^\pm = I^\pm \times \T \times \left[\tfrac{q_m \eta_m}{4}, \tfrac{3q_m \eta_m}{4}\right],\hone h^\pm = \eta_m^{-1} \pm 1,\]
with $\eta_m$ as defined in Lemma \ref{thm: prop_polynomials}. We claim that for $m$ sufficiently large
\[\twr^\pm = \twr( F^\pm, h^\pm)\]
 are well defined Rokhlin towers for the map $\transf{\varphi}$ satisfying the desired conditions. We will prove this only for $F^-$ as the proof for $F^+$ is analogous. \\
 
 First, notice that $h^- > N$ by our initial hypothesis on $m$. Suppose by contradiction that there exists $0 < k < h^-$ for which $T^k(F^-) \cap F^- \neq \emptyset$. Then there exist $(x,y,z) \in F^-$ such that for $n = \pi(k, x, y, z)$
\begin{equation}
\label{eq: contradiction_assumptions}
x + n\Omega \in I^-, \hone 0\leq k - S_{n}^{R_\omega}\varphi(x,y) \leq \frac{q_m\eta_m}{2}.
\end{equation}
Let $0 < c < 1$ such that $\varphi + P > c$. Notice that $n \leq hc^{-1}$ and take $m$ sufficiently large so that $hc^{-1} < q_{m+1}$. As $x + n\Omega \in I^-$ then 
\[\vvvert n \omega \vvvert \leq \frac{4\mu}{q_m}\] which implies $n > q_{m - 1}$. Since $n \leq ha^{-1} < q_{m + 1}$ it follows that $n = l q_m$ for some $l \in \N$. By definition of $\eta_m$ if follows that $l < e^{q_m}$ for $m$ sufficiently large. \\

\textbf{Claim.} For all $x \in I^-,$ $y \in \T,$ and $0 < l < e^{q_m}$
 \begin{equation}
 \label{eq: Birkhoff_sum_approximation}
 |S_{lq_m}^{R_\omega} \varphi (x,y) - lq_m + lq_m\eta_m| < \frac{1}{e^{2q_m}}
 \end{equation}
 for $m$ sufficiently large. 
\begin{proof}[Proof of the Claim] We decompose $\varphi$ as in (\ref{Decomposition}) and calculate each Birkhoff sum separately. Since $P$ and $L_m$ are polynomials of degree less than $q_{m - 1},$ and $q_m$ respectively, by Corollary \ref{thm: bounds_polynomials} there exist a positive constant $C_0$ such that
\[ \|S_{lq_m}^{R_\omega}P\|_{C^0} \leq \frac{lC_0}{q_{m + 1}}, \hone \|S_{lq_m}^{R_\omega}L_m\|_{C^0} \leq \frac{lC_0q_m}{q_{m + 1}}.\]
Given $x \in I^-$ and $0 < l < e^{q_m}$ 
\[ \left|\{ q_m(x + l\alpha)\} - \frac{3}{4}\right| < 2\mu + \frac{l}{q_{m+1}} < 3\mu\]
for sufficiently large $m$. Thus by condition $3$ of Lemma \ref{thm: prop_polynomials}
\begin{align*}
|S_{lq_m}^{R_\omega}P_{\mu, m}(x) + lq_m\eta_m| \leq \dfrac{le^{3q_m/4}}{q_{m+1}}.
\end{align*}
Notice that $\| H_m\|_{C^0} \leq 2e^{-q_m'}$ for $m$ sufficiently large. Let $x \in I^-$, $y \in \T$ and $0 < l < e^{q_m}$. Then
\begin{align*}
 |S_{lq_m}^{R_\omega}\varphi (x,y) - lq_m + lq_m\eta_m |& \leq |S_{lq_m}^{R_\omega}P_{\mu, m}(x,y) + lq_m \eta_m| \\
 & \quad + |S_{lq_m}^{R_\omega}P| + |S_{lq_m}^{R_\omega}L_m| + |S_{lq_m}^{R_\omega}H_m| \\ 
 & \leq \dfrac{le^{3q_m/4}}{q_{m+1}} + \frac{lC_0}{q_{m + 1}} + \frac{lC_0q_m}{q_{m + 1}}+ \dfrac{2lq_m}{e^{q_m'}} \\
 & < \dfrac{1}{e^{2q_m}}
 \end{align*}
 for sufficiently large $m$.
 \end{proof}
 
By (\ref{eq: contradiction_assumptions}) and the previous Claim
\begin{equation}
\label{eq: bounds_k}
lq_m - lq_m\eta_m- e^{-2q_m} < k < lq_m- lq_m\eta_m + e^{-2q_m} + \frac{q_m\eta_m}{2}
\end{equation}
Since $k \leq \eta_m^{-1} - 2$ the LHS of the previous equation yields to
\[ l q_m \eta_m \leq 1 - \frac{\eta_m}{1 - \eta_m}(1 - e^{-2q_m}).\]
Thus, by applying the previous inequality in the LHS of (\ref{eq: bounds_k}) it follows that
\[ k > lq_m - 1\]
for $m$ sufficiently large. But the RHS of (\ref{eq: bounds_k}) implies 
\[ k < lq_m\]
for $m$ sufficiently large, which is a contradiction since $k$ is a natural number. Therefore $\twr^- = \twr(F^-, h^-)$ is a well defined Rokhlin tower for $\transf{\varphi}$. Furthermore 
 \[ \mu_{\twr^-} = h^- \mu(F^-) = \dfrac{1 - \eta_m}{\eta_m} \dfrac{3\eta_m\mu}{2} > \mu.\]
Let $(x,y,z) \in F^{-}$. Taking $l = (\eta_m q_m)^{-1}$ in (\ref{eq: Birkhoff_sum_approximation}) we obtain
\[ \left| h^- - S^{R_\omega}_{h^- + 1} \varphi(x, y)\right| < \dfrac{1}{e^{2q_m}}.\]
Hence $\pi(h^-, x, y, z) =h^- + 1 = \eta_m^{-1}$ and
\begin{align*}
\left| \map{\varphi}^{h^-}(x,y,z) - (x,y,z)\right| &= \left(\vvvert \eta_m^{-1}\Omega \vvvert, \vvvert \eta_m^{-1} \Omega' \vvvert, h^- - S^{R_\omega}_{h^- + 1} \varphi(x, y) \right).
\end{align*}
Since $\vvvert \eta_m^{-1}\Omega \vvvert \leq \frac{\eta_m^{-1}}{q_{m + 1}}$ it easily follows that 
\[\rho_{\twr^-} < \dfrac{\mu}{2(h^+)^2}\]
for sufficiently large $m$. It remains to show $d(\twr^-, \Pa) \leq \beta$. For this consider an open cover of $\T^3$ 
\[\mathcal{U} = \{ \mathring{P} \mid P \in \Pa, \mathring{P} \neq \emptyset \} \cup \{ U\}\]
where $U$ is an open set of $\mu_{\varphi_0 + P}$-measure less than $\frac{\beta}{2}$ and let $a$ be the Lebesgue number of this cover. The existence of such cover follows directly from the hypotheses on $\Pa$. We claim that for $m$ sufficiently large the diameter of every level of the towers $\twr^-$ is smaller than $a$. By (\ref{eq: Birkhoff_sum_approximation}) for all $0 \leq l < e^{q_m}$ 
\begin{align*}
diam(\map{\varphi}^{lq_m}(F^-)) &\leq q_m\eta_m+ \dfrac{1}{e^{2q_m}} + \dfrac{3\mu}{q_m}\| \varphi'\|_{C^0} \\
& \leq \dfrac{2q_m^2}{e^{q_m}} + \dfrac{1}{e^{2q_m}} + \dfrac{6\mu}{q_m} \\
& \leq \dfrac{10\mu}{q_m}
\end{align*}
for $m$ sufficiently large. By definition of $\twr^-$ and by Corollary \ref{thm: bounds_polynomials}, there exists a positive constant $C$, such that for all $0 < r < q_m$
\begin{align*}
 diam( & \map{\varphi}^{lq_m + r}(F^-)) \leq diam(\map{\varphi}^{lq_m}(F^-))\Big( 1 + \max_{0 \leq i < q_m} \| S_i\varphi' \|_{C^0} \Big) \\
& \leq \dfrac{10\mu}{q_m} \bigg( 1 + \max_{0 \leq i < q_m} \| S_iP' \|_{C^0} + \max_{0 \leq i < q_m} \| S_iL_{m - 1}' \|_{C^0} \\
& \quad + q_m\|X_{m - 1} + Y_{m - 1} + P_{\mu, m} + H_m\|_{C^1} \bigg) \\ 
& \leq \dfrac{10\mu}{q_m} \left( 1 + C + Cq_{m - 1} + q_m\left( \frac{q_{m-1}}{e^{q_{m - 1}}} + \frac{q_{m-1}'}{e^{-q_{m - 1}'}} + \dfrac{q_m^4}{e^{q_m}} + \dfrac{2q_m}{e^{q_m'}}\right) \right) \\
& \leq a
\end{align*}
for $m$ sufficiently large. Therefore every level of the tower $\twr^-$ is contained in at least one element of the open cover $\mathcal{U}$. By definition of $\mathcal{U}$ it easily follows that
\[ \Delta(\twr^-, \Pa) < \meas{\varphi}(U)\]
 Since $\mu_{\omega, P + \varphi_0} (U) < \frac{\beta}{2}$ we have $\meas{\varphi}(U) < \beta$ for $m$ sufficiently large. 
 \end{proof}

We now have all the tools to prove Theorem \ref{thm: mixing_LB_existence}. 

\begin{proof}[Proof of Theorem \ref{thm: mixing_LB_existence}] 
Let $\varphi \in C^\infty_+(\T^2)$ with average equal to one, $r \in \N$ and $\epsilon > 0$. We will define $\psi \in C^\infty_+(\T^2)$ with average equal to one such that 
\[\| \varphi - \psi\|_{C^r} < \epsilon,\]
 $\transf{\psi}$ is mixing and $\cartesian{\psi}$ is Loosely Bernoulli. We will obtain $\psi$ as the limit of a recursive procedure. By Proposition \ref{thm: mixing_density} there exists a trigonometric polynomial $P_0$ with zero average such that $\psi_0 = \varphi_0 + P_0 \in \MixColl$ and 
 \[\| \varphi - \psi_0\|_{C^r} < \frac{\epsilon}{2}.\] 
We define a sequence $(\Pa_n)_{n \in \N}$ of refining partitions as follows. For all $n \in \N$, $\Pa_n$ is the partition of $\T^3$ into $2^{3n}$ open cubes of the form
\[ \left(\tfrac{i}{2^n}, \tfrac{i + 1}{2^n}\right) \times \left(\tfrac{j}{2^n}, \tfrac{j + 1}{2^n}\right) \times \left(\tfrac{k}{2^n}, \tfrac{k + 1}{2^n}\right), \]
for some $i, j, k \in \N$, together with a zero Lebesgue measure set given by the union of the boundaries of the cubes. The order of the elements of each partition is not relevant. Clearly $\Pa_n \rightarrow \mathcal{B}(\T^3)$. Let $\mu_n = (n + 1)^{-\frac{1}{4}}$ and define 
\[ \alpha_0 = 0, \hone \alpha_{n + 1} = \alpha_{n} + \delta_n, \hone \delta_n = \frac{\mu_n^4}{4}\left( \frac{1 - \alpha_n}{10} \right)^2.\]
 for all $n \in \N$. Propositions \ref{thm: open_property}, \ref{thm: prop_polynomials}, \ref{thm: polynomial_perturbation} imply the following. \\

\begin{lem}
\label{thm: recursive_construction}
There exist closed sets $\{F_n^\pm\}_{n \geq 1} \subset \mathcal{E}$, sequences $\{ h_n^\pm\}_{n \geq 1} \subset \N$, $\{ N_{n, k}\}_{n, k \in \N} \subset \N$, $\{m_n\}_{n \geq 1} \subset \N$, $\{\epsilon_n\}_{n \in \N} \subset \R_+$ with $m_n$ increasing and $\epsilon_n$ decreasing, such that for all $0 \leq k < l \leq n$, denoting
\[ P_l = P_{\mu_l, m_l} - X_{m_l},\] with $P_{\mu_l, m_l}$ as in Proposition \ref{thm: prop_polynomials}, and 
\[ \psi_n(x, y) = \varphi_0(x, y) + P_0(x, y) + P_1(x) +\dots + P_n(x), \]
the following holds:
\begin{enumerate}
\item \label{mixing} $\psi_n \in \MixColl.$
\item \label{smallness} $\| P_n\|_{C^{r + n}} \leq \epsilon_{n - 1} \leq \tfrac{1}{2^n}.$
\item $ \label{height_condition_verified} h_n^+ = h_n^- + 2,$ and $h_{n}^+ h_{n}^- > N_{k, n - k}.$
\item \label{thm_hypotheses} For all $\psi \in \MixColl$ obeying
\[ \| \psi - \psi_n\|_{C^0} \leq \epsilon_n,\]
$\twr_l^\pm = \twr(\map{\psi}, F_l^\pm, h_l^\pm)$ are well defined Rokhlin towers, $\delta_{l}$ - monochromatic with respect to $\Pa_k$ obeying
\[\mu_{\twr_l}^\pm \geq \frac{1}{l^{1/4}}, \hone \rho_{\twr^\pm_l} < \frac{\mu_{\twr_l}}{2(h^{+}_{_l})^2}.\]
Moreover, $(\map{\psi},\Pa_k)$ and $\mathfrak{T} = \{ \twr_i\}_{k \leq i < n}$ verify the hypotheses of Proposition \ref{thm: LB_criterion_towers_product}. For $0 \leq i \leq n - k$, the numbers $N_i(\map{\psi},$ $\Pa_k, \twr_k^\pm, \dots, \twr_{k + i - 1}^\pm)$ given by Proposition \ref{thm: LB_criterion_towers_product} can be taken as $N_{k, i}.$
\end{enumerate}
\end{lem}
 Let us assume that the Lemma has been proven and let
 \[ \psi = \lim_{n \rightarrow \infty} \psi_n,\]
 which is a well defined $C^\infty$ function obeying
 \[ \| \varphi - \psi\|_{C^r} < \epsilon.\]
 We can suppose WLOG that $\psi \in C^\infty_+(\T^2)$. By Proposition \ref{thm: prop_polynomials}, $\psi \in \MixColl$ and thus $\transf{\psi}$ is mixing. By (\ref{thm_hypotheses}) in the Lemma, the transformation $\map{\psi}$ and the collection $\{\twr_n^\pm\}_{n \geq 1}$ verify the hypotheses of Theorem \ref{thm: LB_criterion_transformation_product}. Furthermore, by (\ref{height_condition_verified}) and (\ref{thm_hypotheses}) in the Lemma and by the Remark in Theorem \ref{thm: LB_criterion_transformation_product} it follows that $\cartesian{\psi}$ is loosely Bernoulli. 
 \end{proof}
 
 \begin{proof}[Proof Lemma \ref{thm: recursive_construction}] We prove the Lemma by induction. Define $\epsilon_0 = \frac{\epsilon}{4}$, $N_{0, 0} = 0$ and let $n \geq 1$. Suppose $N_{k, i}$ has been defined for all $0 \leq k < n,$ $0 \leq i < n - k$ and that the other sequences in the statement have been defined up to their $(n - 1)$-th term. For all $0 \leq k \leq n$ define
\[N_{k, n - k} = N_{n - k}(\map{\psi_{n - 1}}, \Pa_k, \twr(\map{\psi_{n - 1}}, F_k^\pm, h_k^\pm), \dots, \twr(\map{\psi_{n - 1}}, F_{n - 1}^\pm, h_{n - 1}^\pm))\]
with $N_{n - k}$ as in Proposition \ref{thm: LB_criterion_towers_product}. 
 Applying Proposition \ref{thm: polynomial_perturbation} with 
\[ P = P_0 + \dots + P_{n - 1}, \hone \Pa = \Pa_n, \hone \beta = \delta_n, \] 
\[\mu = \mu_{n}, \hone \epsilon = \epsilon_{n - 1}, \hone N = \frac{2}{\delta_{n}} \max_{0 \leq k \leq n} N_{k, n - k}\]
there exists $m_{n } \in \N$ such that for $P_{n} = P_{\mu_{n}, m_{n}} - X_{m_{n}}$ 
 \[ \psi_{n} = \psi_{n - 1} + P_{n} \in C^\infty_+(\T^2), \hone \| P_{n} \|_{C^{r + n}} < \epsilon_{n - 1},\]
and the system $\transf{\psi_{n}}$ admits Rokhlin towers $\twr^\pm = \twr(\map{\psi_n}, F_{n}^\pm, h_n^\pm)$, $\delta_n$-monochromatic with respect to $\Pa_n$, with closed floor sets such that
 \[ h_n^+ = h_n^- + 2 > N, \hone \mu_{\twr^\pm} > \mu_{n}, \hone \rho_{\twr^\pm} < \frac{\mu_{n}}{2(h^{+}_{_n})^2}.\]
By Proposition \ref{thm: mixing_density}, $\psi_n$ satisfies the first assertion. By construction the second and third assertion also hold for $P_n$ and $h_n^\pm$ respectively. Thus, it remains to define $\epsilon_n$ such that the fourth assertion holds. Let $0 \leq k < l \leq n $. Since
\[ \|\psi_n - \psi_{n - 1} \|_{C^0} < \epsilon_{n - 1}\]
it follows from (\ref{thm_hypotheses}) in the Lemma that
\[N_{k, n - k} = N_{n - k}(\map{\psi_{n}}, \Pa_k, \twr(\map{\psi_{n}}, F_k^\pm, h_k^\pm), \dots, \twr(\map{\psi_{n}}, F_{n - 1}^\pm, h_{n - 1}^\pm)).\]
By (\ref{thm_hypotheses}) in the Lemma, if $k < n - 1$ the transformation $(\map{\psi_n}, \Pa_{k})$ and the collection $\mathcal{T} = \{ \twr(\map{\psi_n}, F_i^\pm, h_i^\pm)\}_{k \leq i < n}$ verify the hypotheses of Proposition \ref{thm: LB_criterion_towers_product}. By construction, if $k = n - 1$ then $(\map{\psi_n}, \Pa_{n - 1})$ and $\mathcal{T} = \{ \twr(\map{\psi_n}, F_n^\pm, h_n^\pm)\}$ verify the hypotheses of Proposition \ref{thm: LB_criterion_towers_product}. By Proposition \ref{thm: open_property} and by continuity, there exist 
\[0 < \epsilon_{n} \leq \frac{\epsilon}{2^{n + 2}}\]
such that the conclusions in (\ref{thm_hypotheses}) hold for any $\psi \in \MixColl$ satisfying 
\[\| \psi - \psi_n\| \leq \epsilon_n.\]
This completes the proof.
\end{proof}

\section{Appendix}

\begin{lem}
\label{thm: lemma_cohomological}
Let $\omega \in \R^d$ such that $(1, \omega)$ is non-resonant. Given a trigonometric polynomial $P: \T^d \rightarrow \R$ with zero average and of degree $n \geq 1$ there exists a trigonometric polynomial $Q$ of degree $n$ such that
\begin{equation} 
\label{eq: cohomological}
P = Q \circ R_\omega - Q,
\end{equation}
where $R_\omega: \T^d \rightarrow \T^d$ denotes the toral translation of rotation vector $\omega$. In particular, the Birkhoff sums of $P$ with respect to $R_\omega$ obey
\begin{equation}
\label{eq: bound_birkhoff_sums}
\left\| S^{R_\omega}_m P \right\|_{C^r} \leq \min \left\{ 2\| Q\|_{C^r}, \vvvert m\omega\vvvert \| Q\|_{C^{r + 1}} \right\}
\end{equation}
 for all $m \geq 1.$ Furthermore, there exists a constant $C(d, r)$ such that
\[\| Q\|_{C^r} \leq C \| P\|_{C^{d + r + 1}} \max_{0 < |k|_1 \leq n} \vvvert \langle k, \omega \rangle \vvvert ^{-1}.\]

\end{lem}
\begin{proof}
Consider the Fourier expansion of $P$ 
\[ P(\theta) = \sum_{ |k|_1 \leq n} p_k e^{2\pi i k \cdot \theta}\]
and define
\begin{equation}
\label{eq: cohomological_solution}
Q(\theta) = \sum_{ |k|_1 \leq n} \frac{p_k}{e^{2\pi i k \cdot \theta} - 1} e^{2\pi i k \cdot \theta}.
\end{equation}
Then $Q$ satisfies (\ref{eq: cohomological}). Equation (\ref{eq: bound_birkhoff_sums}) is a direct consequence of (\ref{eq: cohomological}). Recall that given $r \in \N$ there exists a constant $C(d, r)$ such that for any $f \in C^\infty(\T^d ,\R)$ with Fourier coefficients $f_k$ we have
\[ C^{-1} |k|_1^{r} |f_k| \leq \| f\|_{C^r} \leq \sup_{k \in \N^d} C|k|_1^{r + d + 1} |f_k|.\]
Since
\[ |e^{2\pi i k \cdot \theta} - 1| \geq 2\sin (\pi \vvvert k \cdot \omega \vvvert) \geq \vvvert k \cdot \omega \vvvert \]
the bound for the $C^r$ norm of $Q$ now follows from (\ref{eq: cohomological_solution}). 
\end{proof}

\textbf{Remark:} For $\omega \in \R \setminus \mathbb{Q}$ with return times $(q_n)_{n \in \N}$
\[q_n \leq \max_{0 < |k|_1 < q_{n}} \vvvert k\omega \vvvert ^{-1} = \vvvert q_{n - 1} \alpha \vvvert^{-1} \leq 2q_{n}.\]

\begin{cor}
\label{thm: bounds_polynomials}
Let $\omega \in \R^d$ such that $(1, \omega)$ is non-resonant. Given $r \in \N$ there exists a constant $C(d, r)$ such that for any trigonometric polynomial $P: \T^d \rightarrow \R$ with zero average and for all $m \geq 1$
\[ \left\| S^{R_\omega}_m P \right\|_{C^r} \leq C\| P\|_{C^{r + d + 1}} \vvvert m \omega \vvvert \max_{0 < |k|_1 < \deg(P)} \vvvert \langle \omega, k \rangle \vvvert^{-1} \]
In particular, if $d = 1$, $\omega \in \R \setminus \mathbb{Q}$ with return times $(q_n)_{n \in \N}$ and $\deg(P) < q_{m_0}$ then
\[ \left\| S^{R_\omega}_{q_{m - 1}} P \right\|_{C^r} \leq 2C\frac{q_{m_0}}{q_m}\| P\|_{C^{r + 2}}\]
 for all $m \geq 1$.
\end{cor}

The construction in Proposition \ref{thm: prop_polynomials} is analogous of the polynomials families defined in \cite[Proposition 3.5]{fayad_smooth_2006} and in \cite[Proposition 4]{fayad_continuous_2019}. The proof of the last part of Proposition \ref{thm: prop_polynomials} is completely analogous to the proofs of \cite[Theorem 3.1]{fayad_smooth_2006} and \cite[Theorem 3]{fayad_continuous_2019} which rely on the mixing criterion in Proposition \ref{thm: mixing_criterion}. We will briefly sketch the proofs here. For more details we refer the interested reader to \cite{fayad_analytic_2002}, \cite{fayad_smooth_2006}, \cite{fayad_continuous_2019}. 

\begin{proof}[Proof Proposition \ref{thm: prop_polynomials}]
Let $r, n \in \N$ and $0 < \mu < \frac{1}{4}$. Let $\widetilde{\xi}_n: \R \rightarrow \R$ be the $\frac{1}{q_n}$ periodic odd function defined by
\begin{equation*}
\widetilde{\xi}_n(x) = \left\{ 
\begin{array}{ccc} 
x & \text{ for } & x \in \big[0, \frac{1}{4q_n} - \frac{\mu}{q_n} \big] \Big. \\
 \frac{1 - 4\mu}{4q_n} & \text { for } & \big|x - \frac{1}{4q_n}\big| < \frac{\mu}{q_n} \Big.\\
 -x + \frac{1}{2q_n} & \text{ for } & \big[\frac{1}{4q_n} + \frac{\mu}{q_n}, \frac{1}{2q_n} \big] \Big.
 \end{array}
\right. .
\end{equation*}
Fix 
\[\frac{q_ne^{-q_n}}{2} \leq \eta_n \leq \frac{3q_ne^{-q_n}}{4}\]
 such that $\frac{1}{\eta_n} \in 2q_n\N$ (which is always possible for $n$ sufficiently large) and define
\[ \xi_n(x) = \frac{4q_n\eta_n}{1 - 4\mu} \widetilde{\xi}_n(x).\]
Let $K: \R \rightarrow \R$ be an even, positive, smooth bump function obeying 
\[Supp(K) \subset (-1, 1), \hspace{1cm} \int_\R K (x) dx = 1.\]
 Define $K_n(x) = n^2 q_n K(n^2q_n)$. Clearly 
 \[ Supp(K_n) \subset \left(\frac{-1}{n^2q_n}, \frac{1}{n^2q_n}\right), \hspace{1cm} \int_\R K_n(x) dx = 1. \]
Let 
\[ X = K_n \ast \xi_n\]
and notice that $X$ is a well defined $\tfrac{1}{q_n}$ periodic odd function. Define $P_{\mu, n}$ as the truncation of the Fourier series of $X = K_n \ast \xi_n$ to degree $q_{n+1} - 1$, namely
\[ P_{\mu, n}(x) = \sum_{k = -q_{n+1}+1}^{q_{n+1}-1} \widehat{X}(k) e^{2\pi i k x}.\]
Let us denote 
\[I = \left\{ x \in \T \, \Big| \, \big|\{q_nx\} + \tfrac{1}{4} \big| < 3\mu \right\}, \hone I' = \left\{ x \in \T \, \Big| \, \big|\{q_nx\} + \tfrac{3}{4}\big| < 3\mu \right\},\]
\[ J = \left\{ x \in \T \, \Big| \, \big| \, \{q_nx\} + \tfrac{1}{4} \big| > 4\mu \right\}, \hone J' = \left\{ x \in \T \, \Big| \, \big| \, \{q_nx\} + \tfrac{3}{4} \big| > 4\mu \right\}.\]
Notice that the sets $I \cup I'$ and $J \cup J'$ correspond to the points satisfying the inequalities in assertions (3) and (4) respectively. We have
\[ X|_I = \eta_n, \hone X|_{I'} = -\eta_n,\]
\[ X'|_J = \frac{4q_n\eta_n}{1 - 4\mu} , \hone X'|_{J'} = - \frac{4q_n\eta_n}{1 - 4\mu}.\]
In particular
\[ \big| X'|_{J \cup J'} \big| \geq 2q_n^2e^{-q_n}.\]
 By definition of $X$
\begin{align*}
\| X\|_{C^r} & \leq \eta_n (n^{2}q_n)^{r+1}\| K\|_{C^r} \\
& \leq n^{2r+1}q_n^{r+2}e^{-q_n}\| K\|_{C^r} \\
& \leq e^{-4q_n / 5}
\end{align*}
for sufficiently large $n$. Furthermore
\begin{align*}
 \| X - P_{\mu, n}\|_{C^r} & \leq \sum_{|k| \geq q_{n+1}} (2\pi k)^r |\widehat{X}(k)| \\
& \leq \dfrac{1}{2\pi} \sum_{|k| \geq q_{n+1}} \dfrac{\| X\|_{C^{r+1}}}{k^2} \\
& \leq \frac{e^{3q_n/4}}{q_{n+1}}
\end{align*}
for $n$ sufficiently large. Therefore conditions (2)-(4) are verified by $P_{\mu, n}$ for $n$ sufficiently large. It remains to show that the limit function belongs to $\MixColl$. We do this by applying the mixing criterion in Proposition \ref{thm: mixing_criterion}. 
We check only the first part of Proposition \ref{thm: mixing_criterion} as the second part is verified exactly as in \cite{fayad_analytic_2002}, \cite{fayad_smooth_2006}, \cite{fayad_continuous_2019}. We can express $\psi$ as 
\[ \psi(x, y) = 1 + P(x, y) + \sum_{n = 1}^{+\infty} P_n(x) + e^{-q_n'}\cos(2\pi q_n'y)\]
where $P_n(x)$ is either equal to $e^{-q_n}\cos(2\pi q_nx)$ or to some polynomial $P_{\mu_n, q_n}$, of degree less than $q_{n+1}$, given by Proposition \ref{thm: prop_polynomials} for some $\mu_n > \frac{1}{n}$. Define 
\[I_n = \left\{ x \in \T \, \left| \, \{q_nx\} + \tfrac{1}{2} \pm \tfrac{1}{4} > r_n \right. \right\}\]
where 
\[ 
r_n = \left\{ \begin{array}{lc}
5\mu & \text{ if } P_n = P_{\mu_n, q_n}, \\
n^{-1} & \text{ otherwise. }
\end{array} \right. 
\]
 Let $m \in [e^{2q_n}/2, 2e^{2q_n'}]$ and suppose $\deg(P) < q_{n - 1}$. Decompose 
 \[\psi = 1 + L_n + P_n + H_n\]
 in terms of order less or equal than $q_{n - 1}$, of degree $q_n$, and of order bigger or equal than $q_{n+1}$ respectively. By Corollary \ref{thm: bounds_polynomials} applied to $L_n$ there exists a positive constant $C$ such that 
 \[ \|\partial_x S_m(\psi - P_n)\|_{\T^2} \leq Cq_{n} + 2e^{-q_n'} \leq \dfrac{m}{q_n} \]
for sufficiently large $n$. If $P_n = e^{-q_n}\cos(2\pi q_nx)$ it follows from \cite{fayad_analytic_2002} that 
\[ |\partial_xS_mP_n(x)| \geq 16\dfrac{mq_n}{ne^{q_n}},\]
for all $x \in I_n$. If $P_n = P_{\mu_n, q_n}$, for all $x \in I_n$ and all $0 \leq l \leq 2e^{q_n'}$ 
\[ \left|\{ q_n(x + l\Omega)\} + \frac{1}{2} \pm \frac{1}{4}\right| > 5\mu - \frac{l}{q_{m+1}} > 4\mu\]
for sufficiently large $n$. Thus by condition (4) of Proposition \ref{thm: prop_polynomials}
\[ |\partial_xS_{m}P_n(x)| \geq \dfrac{mq_n^2}{e^{q_n}} \]
for all $x \in I_n$. In both cases the first part of Proposition \ref{thm: mixing_criterion} is verified. Thus $\psi \in \MixColl$. 
\end{proof}

\section*{Acknowledgments}

The author would like to thank Jean-Paul Thouvenot for his encouragement and many helpful discussions. 

\bibliographystyle{acm}
\bibliography{LooselyBernoulli.bib}

\begin{thebibliography}{10}

\bibitem{abramov_entropy_1959}
{\sc Abramov, L.~M.}
\newblock On the entropy of a flow.
\newblock {\em Doklady Akademii Nauk SSSR 128\/} (1959), 873--875.

\bibitem{dye_groups_1959}
{\sc Dye, H.~A.}
\newblock On {Groups} of {Measure} {Preserving} {Transformations}. {I}.
\newblock {\em American Journal of Mathematics 81}, 1 (1959), 119--159.

\bibitem{fayad_rank_2005}
{\sc Fayad, B.}
\newblock Rank one and mixing differentiable flows.
\newblock {\em Inventiones mathematicae 160}, 2 (May 2005), 305--340.

\bibitem{fayad_smooth_2006}
{\sc Fayad, B.}
\newblock Smooth mixing flows with purely singular spectra.
\newblock {\em Duke Mathematical Journal 132}, 2 (Apr. 2006), 371--391.

\bibitem{fayad_continuous_2019}
{\sc Fayad, B., and Qu, Y.}
\newblock Continuous spectrum for a class of smooth mixing {Schr{\"o}dinger}
  operators.
\newblock {\em Ergodic Theory and Dynamical Systems 39}, 2 (Feb. 2019),
  357--369.

\bibitem{fayad_analytic_2002}
{\sc Fayad, B.~R.}
\newblock Analytic mixing reparametrizations of irrational flows.
\newblock {\em Ergodic Theory and Dynamical Systems 22}, 2 (Apr. 2002),
  437--468.

\bibitem{feldman_newk-automorphisms_1976}
{\sc Feldman, J.}
\newblock {NewK}-automorphisms and a problem of {Kakutani}.
\newblock {\em Israel Journal of Mathematics 24}, 1 (Mar. 1976), 16--38.

\bibitem{ferenczi_systemes_1984}
{\sc Ferenczi, S.}
\newblock Syst{\`e}mes localement de rang un.
\newblock {\em Annales de l'I.H.P. Probabilit{\'e}s et statistiques 20}, 1
  (1984), 35--51.

\bibitem{ferenczi_systems_1997}
{\sc Ferenczi, S.}
\newblock Systems of finite rank.
\newblock {\em Colloquium Mathematicum 73\/} (1997), 35--65.

\bibitem{gerber_zero-entropy_1981}
{\sc Gerber, M.}
\newblock A zero-entropy mixing transformation whose product with itself is
  loosely {Bernoulli}.
\newblock {\em Israel Journal of Mathematics 38}, 1 (Mar. 1981), 1--22.

\bibitem{gerber_smooth_2018}
{\sc Gerber, M., and Kunde, P.}
\newblock A smooth zero-entropy diffeomorphism whose product with itself is
  loosely {Bernoulli}.
\newblock {\em arXiv:1803.01926 [math]\/} (Mar. 2018).
\newblock arXiv: 1803.01926.

\bibitem{halmos_approximation_1944}
{\sc Halmos, P.~R.}
\newblock Approximation {Theories} for {Measure} {Preserving}
  {Transformations}.
\newblock {\em Transactions of the American Mathematical Society 55}, 1 (1944),
  1--18.

\bibitem{halmos_general_1944}
{\sc Halmos, P.~R.}
\newblock In {General} a {Measure} {Preserving} {Transformation} is {Mixing}.
\newblock {\em Annals of Mathematics 45}, 4 (1944), 786--792.

\bibitem{janvresse_pascal_2004}
{\sc Janvresse, {\'E}., and de~la Rue, T.}
\newblock The {Pascal} adic transformation is loosely {Bernoulli}.
\newblock {\em Annales de l'Institut Henri Poincare (B) Probability and
  Statistics 40}, 2 (Mar. 2004), 133--139.

\bibitem{kakutani_induced_1943}
{\sc Kakutani, S.}
\newblock Induced measure preserving transformations.
\newblock {\em Proceedings of the Imperial Academy 19}, 10 (1943), 635--641.

\bibitem{katok_combinatorial_2003}
{\sc Katok, A.}
\newblock {\em Combinatorial {Constructions} in {Ergodic} {Theory} and
  {Dynamics}}, vol.~30 of {\em University {Lecture} {Series}}.
\newblock American Mathematical Society, Oct. 2003.

\bibitem{katok_entropy_1967}
{\sc Katok, A.~B.}
\newblock Entropy and approximations of dynamical systems by periodic
  transformations.
\newblock {\em Functional Analysis and Its Applications 1}, 1 (Jan. 1967),
  66--74.

\bibitem{katok_time_1975}
{\sc Katok, A.~B.}
\newblock Time change, monotone equivalence, and standard dynamical systems.
\newblock {\em Doklady Akademii Nauk SSSR 223}, 4 (1975), 789--792.

\bibitem{katok_approximation_1966}
{\sc Katok, A.~B., and Stepin, A.~M.}
\newblock Approximation of ergodic dynamic systems by periodic transformations.
\newblock 1268--1271.

\bibitem{katok_approximations_1967}
{\sc Katok, A.~B., and Stepin, A.~M.}
\newblock Approximations in ergodic theory.
\newblock {\em Akademiya Nauk SSSR i Moskovskoe Matematicheskoe Obshchestvo.
  Uspekhi Matematicheskikh Nauk 22}, 5 (137) (1967), 81--106.

\bibitem{katok_metric_1970}
{\sc Katok, A.~B., and Stepin, A.~M.}
\newblock Metric properties of homeomorphisms that preserve measure.
\newblock {\em Akademiya Nauk SSSR i Moskovskoe Matematicheskoe Obshchestvo.
  Uspekhi Matematicheskikh Nauk 25}, 2 (152) (1970), 193--220.

\bibitem{ornstein_equivalence_1982}
{\sc Ornstein, D.~S., Rudolph, D.~J., and Weiss, B.}
\newblock {\em Equivalence of measure preserving transformations}, vol.~37 of
  {\em Memoirs of the {American} {Mathematical} {Society}}.
\newblock American Mathematical Society, 1982.

\bibitem{ratner_horocycle_1978}
{\sc Ratner, M.}
\newblock Horocycle flows are loosely {Bernoulli}.
\newblock {\em Israel Journal of Mathematics 31}, 2 (June 1978), 122--132.

\bibitem{ratner_cartesian_1979}
{\sc Ratner, M.}
\newblock The cartesian square of the horocycle flow is not loosely
  {Bernoulli}.
\newblock {\em Israel Journal of Mathematics 34}, 1 (Mar. 1979), 72--96.

\bibitem{rohlin_general_1948}
{\sc Rohlin, V.}
\newblock A {\textquotedblleft}general{\textquotedblright} measure-preserving
  transformation is not mixing.
\newblock {\em Doklady Akad. Nauk SSSR (N.S.) 60\/} (1948), 349--351.

\bibitem{stepin_spectrum_1967}
{\sc Stepin, A.~M.}
\newblock The spectrum and approximation of metric automorphisms by periodic
  transformations.
\newblock {\em Akademija Nauk SSSR. Funkcional'nyi Analiz i ego Prilo{\v
  z}enija 1}, 2 (1967), 77--80.

\bibitem{yoccoz_centralisateurs_1995}
{\sc Yoccoz, J.-C.}
\newblock Centralisateurs et conjugaison diff{\'e}rentiable des
  diff{\'e}omorphismes du cercle.
\newblock In {\em Ast{\'e}risque}, no.~231. 1995, pp.~89--242.

\end{thebibliography}

\end{document}